\documentclass[11pt]{article}

\usepackage{a4wide}
\usepackage{setspace}
\usepackage{amsmath,amsthm,amssymb,latexsym}
\usepackage{tikz}
\usetikzlibrary{matrix,arrows}
\usepackage{verbatim}
\usepackage[left=2.5cm,top=3cm,right=3cm,bottom=2.5cm]{geometry}
\usepackage{fancyhdr}

\newcommand{\gap}{\vspace{5mm}}

\let\oldmarginpar\marginpar
\renewcommand\marginpar[1]{\-\oldmarginpar[\raggedleft\footnotesize #1]{\raggedright\footnotesize #1}}

\newcounter{mnotes}
\theoremstyle{plain}

\newtheorem{thm}{Theorem}

\newtheorem{lem}[thm]{Lemma}

\newtheorem{cor}[thm]{Corollary}
\newtheorem{prop}[thm]{Proposition}
\newtheorem{fact}[thm]{Fact}

\theoremstyle{definition}
\newtheorem{dfn}[thm]{Definition}
\newtheorem{exa}[thm]{Example}

\newtheorem{rem}[thm]{Remark}

\newcommand{\cat}[1]{\mathsf{#1}}

\newcommand{\DLat}{\cat{DLat}}
\newcommand{\Filt}{\cat{filt}}

\newcommand{\id}{\cat{id}}
\newcommand{\Idl}{\cat{idl}}

\newcommand{\JSPL}{\cat{JSPL}}
\newcommand{\MSPL}{\cat{MSPL}}
\newcommand{\DSPL}{\cat{DSPL}}
\newcommand{\JSDPL}{\cat{JSDPL}}

\newcommand{\KO}{\cat{KO}}
\newcommand{\KS}{\cat{KS}}
\newcommand{\Lat}{\cat{Lat}}
\newcommand{\pt}{\cat{pt}}

\DeclareMathOperator{\spec}{\cat{spec}\,}

\newcommand{\Top}{\cat{Top}}
\newcommand{\op}{\mathrm{op}}

\newcommand{\B}{\mathbb{B}}

\newcommand{\C}{\mathbb{C}}
\newcommand{\D}{\mathbb{D}}
\newcommand{\E}{\mathbb{E}}
\newcommand{\bF}{\mathbb{F}}
\newcommand{\F}{\mathcal{F}}
\renewcommand{\phi}{\varphi}
\newcommand{\G}{\mathcal{G}}
\renewcommand{\H}{\mathcal{H}}

\newcommand{\I}{\mathcal{I}}

\renewcommand{\L}{\mathbb{L}}
\newcommand{\M}{\mathbb{M}}

\newcommand{\N}{\mathbb{N}}
\renewcommand{\O}{\mathcal{O}}
\renewcommand{\P}{\mathbb{P}}
\newcommand{\Q}{\mathbb{Q}}
\newcommand{\Po}{\mathcal{P}}
\newcommand{\sR}{\mathcal{R}}
\renewcommand{\S}{\mathbb{S}}

\newcommand{\lad}{\dashv}

\renewcommand{\implies}{\Rightarrow}

\newcommand{\limp}{\rightarrow}
\renewcommand{\lor}{\vee}
\renewcommand{\land}{\wedge}

\newcommand{\isom}{\cong}

\title{Duality and canonical extensions for stably compact spaces}
\date{6 September, 2011}\author{Sam van Gool\footnote{\it Institute for Mathematics, Astrophysics and Particle Physics, Radboud Universiteit Nijmegen, P.O. Box 9010, 6500 GL Nijmegen, The Netherlands. E-mail: s.vangool@math.ru.nl}}

\numberwithin{thm}{section}
\begin{document}
\maketitle
\abstract{We construct a canonical extension for strong proximity lattices in order to give an algebraic, point-free description of a finitary duality for stably compact spaces. In this setting not only morphisms, but also objects may have distinct $\pi$- and $\sigma$-extensions.}
\section*{Introduction}
{\it Strong proximity lattices} were introduced, after groundwork of Michael Smyth \cite{Smyth1992}, by Achim Jung and Philipp S\"underhauf \cite{Jung1996}, who showed that these structures are dual to {\it stably compact spaces}, which generalise spectral spaces and are relevant to domain theory in logical form (cf. for example \cite{Abramsky1987} and \cite{Jung2004}). 

The {\it canonical extension}, which first appeared in a paper by Bjarni J\'onsson and Alfred Tarski \cite{Jonsson1952}, has proven to be a powerful method in the study of logics whose operations are based on lattices, such as classical modal logic (\cite{Jonsson1952}, \cite{Jonsson1994}), distributive modal logic (\cite{Gehrke2000}, \cite{Gehrke2004Sahlqvist}), and also intuitionistic logic (\cite{Ghilardi1997}). Canonical extensions are interesting because they provide a formulaic, algebraic description of Stone-type dualities between algebras and topological spaces.

In this paper, we re-examine the Jung-S\"underhauf duality \cite{Jung1996} and put it in a broader perspective by connecting it with the theory of canonical extensions. We now  briefly outline the contents.

Careful study of the duality in \cite{Jung1996} led us to conclude that the axioms for strong proximity lattices were stronger than necessary. The advantage of assuming one axiom less, as we will do here, is that it will become more apparent how the inherent self-duality of stably compact spaces is reflected in the representing algebraic structures. We introduce our definitions and terminology, and discuss the mentioned self-duality, in Section~\ref{sec:representing}.

Our most important contribution is the existence and uniqueness proof of canonical extensions for proximity lattices satisfying one additional strongness condition, which we give in Section~\ref{sec:canext}. In the same section, we start the study of the canonical extensions of morphisms, motivated by the historical cases mentioned above, where extending morphisms made canonical extensions useful.

Then, in Section~\ref{sec:duality}, we report our understanding of strong proximity lattices and the duality from \cite{Jung1996}, which has become clearer after studying it through algebraic methods. In particular, we will show that the duality, as we present it in Section~\ref{sec:representing}, can be understood as an application of a general category-theoretical construction (i.e., {\it splitting by idempotents}) to an earlier, well-known correspondence between continuous functions on spectral spaces and certain relations on the associated distributive lattice.

\section{Representing stably compact spaces}\label{sec:representing}
\subsection{Definitions and examples}
We first recall the definition of stably compact and spectral spaces.
\begin{dfn}
Let $X$ be a topological space. A set $S \subseteq X$ is called {\bf saturated} if it is an intersection of opens, and {\bf compact} if any open cover of $S$ contains a finite subcover.

The space $X$ is called
\begin{enumerate} 
\item {\bf locally compact} if, for any open neighbourhood $U$ of a point $x \in X$, there exists an open set $V$ and a compact set $K$ such that $x \in V \subseteq K \subseteq U$. 

\item {\bf sober} if the assignment $x \mapsto \{U \subseteq X \text{ open} \ | \ x \in U\}$ is surjective onto the set of completely prime filters of the frame of opens of $X$, and {\bf $T_0$} if this assignment is injective.

\item {\bf stably compact} if $X$ is $T_0$, sober, locally compact, and the collection $\KS(X)$ of compact-saturated sets is closed under finite intersections.

\item {\bf spectral} if $X$ is $T_0$, sober and the collection $\KO(X)$ of compact-open sets forms a basis for the open sets which is closed under finite intersections. 
\end{enumerate}
Note that a spectral space is always stably compact.
\end{dfn}
Historically, Grothendieck introduced the term `sober' for the spaces which arose in the work of Papert and Papert \cite{Papert1958}. Stably compact spaces seem to have been first studied by Johnstone \cite{Johnstone1982}, who called them `stably locally compact'. Spectral spaces are much older and were first studied by Stone \cite{Stone1937}. The term `spectral' was introduced by Hochster \cite{Hochster1969}, who proved that spectral spaces are also exactly the spaces which arise as the Zariski spectra of commutative rings. The reader is referred to the historical notes of Chapters II and VI of \cite{Johnstone1982} for more details.

\gap

Let us fix some notation: given a lattice $\L$, we denote by $\L^\op$ the {\bf opposite lattice}, which has the same underlying set as $\L$, but the opposite order (i.e., the operations $\land$ and $\lor$, and $\top$ and $\bot$ are interchanged). For a relation $R \subseteq A \times B$, we define the {\bf converse relation} $R^{-1} \subseteq B \times A$ by $b\, R^{-1} \, a$ iff $aRb$, and we write, for $A' \subseteq A$, $R[A'] := \{b \in B: \exists a \in A' : aRb\}$. The {\bf composition} of relations $R \subseteq A \times B$ and $S \subseteq B \times C$ is written as $a \, R \circ S \, c$, which is defined to hold iff there exists $b \in B$ s.t. $aRb$ and $bSc$.

\gap

Jung and S\"underhauf \cite{Jung1996} defined ``strong proximity lattices'' to obtain algebraic structures dual to stably compact spaces. Note that our definition is more general than the one in \cite{Jung1996}: in particular, our proximity lattices are not assumed to be distributive, and we split the property of a proximity lattice being ``strong'' into ``join-strong'' and ``meet-strong''.

\begin{dfn}
A {\bf proximity lattice} is a pair $(\L,R)$, where $\L = (L, \vee, \wedge, \bot, \top)$ is a lattice and $R \subseteq L \times L$ is a relation satisfying the following axioms:
\begin{enumerate}
\item $R \circ R = R$,
\item For any finite set $A \subseteq L$ and $b \in L$, $\bigvee A \, R \, b \iff \forall a \in A\, aRb$.
\item For any finite set $B \subseteq L$ and $a \in L$, $a \, R \, \bigwedge B \iff \forall b \in B\, aRb$.
\end{enumerate}
A proximity lattice is called {\bf join-strong} if, furthermore,
\begin{enumerate}
\item[4.] For any finite set $B \subseteq L$ and $a \in L$, if $a \, R \, \bigvee B$, then there is a finite set $B' \subseteq R^{-1}[B]$ such that $a \, R \, \bigvee B'$.
\end{enumerate}

Dually, a proximity lattice is called {\bf meet-strong} if
\begin{enumerate}
\item[5.] For any finite set $A \subseteq L$ and $b \in L$, if $\bigwedge A \, R \, b$, then there is a finite set $A' \subseteq R[A]$ such that $\bigwedge A' \, R \, b$.
\end{enumerate}

A proximity lattice is {\bf doubly strong} if it is both join-strong and meet-strong. 

We call a relation $R$ on a lattice {\bf increasing} if $aRb$ implies $a \leq b$. In the original definition of Jung and S\"underhauf, it was emphasized that the relation $R$ of a proximity lattice does not need to be increasing. However, as we will see below, the assumption that the relation {\it is} increasing does not change the category of proximity lattices, up to equivalence, and makes the ensuing theory quite a bit cleaner and easier to present. We will come back to this point in Remark~\ref{rem:increasing}.

This definition of proximity lattice is close to, but a bit more general than that of Jung and S\"underhauf \cite{Jung1996}. Note in particular that the `proximity lattices' of Jung and S\"underhauf are what we will call `distributive proximity lattices'. The idea of using distributive proximity lattices to represent stably compact spaces is already present in Smyth \cite{Smyth1992}, but he did not assume both `strong' axioms (4) and (5) for his structures.
\end{dfn}
It follows directly from the definitions that $((\L^\op)^\op, (R^{-1})^{-1}) = (\L,R)$, and that if $(\L,R)$ is a join-strong proximity lattice, then $(\L^\op,R^{-1})$ is a meet-strong proximity lattice. We will come back to the topological meaning of this order duality in Subsection~\ref{subsec:cocompact}.

\begin{exa}
\label{exa:openbasis}
Let $X$ be a stably compact space. Let $\D$ be a basis for the open sets which is closed under finite intersections and finite unions. Note that $\D$ with the inclusion order is a distributive lattice. Define the relation $R \subseteq \D \times \D$ by
\[ dRe \iff \text{ there exists a $k \in \KS(X)$ such that } d \subseteq k \subseteq e. \]
We call $(\D,R)$ an {\bf open-basis presentation} of the space $X$.

Dually, if $\E$ is a `basis' for the compact saturated sets of $X$ (i.e., every compact saturated set $K$ of $X$ is an intersection of elements from $\E$) which is closed under finite unions and intersections, we regard it as a distributive lattice with the converse inclusion order. We then define the relation $S \subseteq \E \times \E$ by 
\[ kSl \iff \text{ there exists an open set $u$ such that } k \supseteq u \supseteq l, \] and call $(\E,S)$ a {\bf compsat-basis presentation} of $X$. 

\begin{fact}\label{fact:openbasisjspl}
\begin{enumerate}
\item An open-basis presentation of a stably compact space is a join-strong proximity lattice, which is furthermore increasing and distributive.

\item A compsat-basis presentation of a stably compact space is a meet-strong proximity lattice, which is furthermore increasing and distributive.
\end{enumerate}
\end{fact}
\begin{proof}
In both items, it is not hard to check that all the axioms for a proximity lattice are satisfied. The arguments for join- and meet-strongness are essentially the same as those given in the proof of Theorem 23 in \cite{Jung1996}.
\end{proof}
\end{exa}

\begin{exa}\label{exa:doublerep}
To get a {\it doubly} strong proximity lattice representing a stably compact space, we can construct a lattice of pairs of open and compact-saturated sets, as was done in Section 6 of \cite{Jung1996}. We briefly recall this construction.

Let $(\D,R)$ and $(\E,S)$ be an open-basis and a compsat-basis presentation of a stably compact space $X$. Let $\bF$ be the sublattice of the lattice $\D \times \E^\op$, consisting of those pairs $(d,e)$ for which $d \subseteq e$ as subsets of $X$. Define the relation $T$ on $\L$ by $(d,e)T(d',e')$ iff $e \subseteq d'$ as subsets of $X$.

\begin{fact}[Theorem 23, \cite{Jung1996}]
$(\bF,T)$ is a doubly strong distributive proximity lattice.
\end{fact}
\end{exa}

\begin{exa}
Note that any basis for a space $X$ which is closed under finite unions must contain all compact-open sets of the space. If $X$ is a spectral space, then we can take the basis $\D$ consisting just of compact-open sets. The relation $R$ from Example~\ref{exa:openbasis} then coincides with the lattice order, and $(\D,R)$ is doubly strong.
\end{exa}

An example with a more algebraic flavour is the following.
\begin{exa}
Let $X$ be a set of variables ({\it generators}) and $E$ a set of pairs of lattice terms  ({\it relations}) in the variables from $X$. The {\bf lattice $\L_{(X,E)}$ presented by $(X,E)$} is the quotient of the free lattice $\F(X)$ on the variables $X$ by the smallest congruence containing $E$. 

Now suppose $S$ is a relation on $\L_{(X,E)}$ which makes $(\L_{(X,E)},S)$ into a proximity lattice. We have the natural homomorphism $h: \F(X) \to \L_{(X,E)}$, which induces a relation $R$ on $\F(X)$ by $aRb$ iff $h(a) S h(b)$. Then $(\F(X),R)$ is also a proximity lattice. Moreover, if $S$ satisfies either of the strong axioms, then so does $R$. Notice also that $R$ is not necessarily an increasing relation, whereas the proximity relations in the previous examples were.
\end{exa}

This last example shows that there are proximity lattices in which $R$ is not increasing, but cf. Remark~\ref{rem:increasing} and Proposition~\ref{prop:increasingisom}.

\subsection{Morphisms}
It should be clear from the above examples that distributive proximity lattices which look very different may present the same stably compact space; for example, two bases for the same space do not even need to have the same cardinality. Thus, two proximity lattices may present the same space, even if their underlying lattices are not isomorphic. However, we do want  proximity lattices which present the same space to be isomorphic in the category of proximity lattices. Consequently, it should come as no surprise that the notion of morphism for proximity lattices needs to be quite lax.

\begin{dfn}
Let $(\L,R)$ and $(\M,S)$ be proximity lattices. A {\bf proximity relation} between $(\L,R)$ and $(\M,S)$ is a relation $G \subseteq L \times M$ which satisfies the following conditions:
\begin{enumerate}
\item $G \circ S = G$,
\item $R \circ G = G$,
\item For any finite set $A \subseteq L$ and $b \in M$, $\bigvee A \, G \, b \iff \forall a \in A \, a G b.$
\item For any finite set $B \subseteq M$ and $a \in L$, $a \, G \, \bigwedge B \iff \forall b \in B\, a G b.$
\end{enumerate}
The relation $G$ is called {\bf join-approximable} if, furthermore
\begin{enumerate}
\item[4.] For any finite set $B \subseteq M$ and $a \in L$, if $a\, G \, \bigvee B$, then there is a finite set $A \subseteq G^{-1}[B]$ such that $a \, R \, \bigvee A$.\end{enumerate}

Dually, $G$ is called {\bf meet-approximable} if
\begin{enumerate}
\item[5.] For any finite set $A \subseteq L$ and $b \in M$, if $\bigwedge A \, G \, b$, then there is a finite set $B \subseteq G[A]$ such that $\bigwedge B \, S \, b$.
\end{enumerate}
A {\bf proximity morphism} from a proximity lattice $(\L,R)$ to a proximity lattice $(\M,S)$ is a relation $H \subseteq L \times M$ such that the relation $H^{-1} \subseteq M \times L$ is a proximity relation. If $H^{-1}$ is furthermore join-approximable (meet-approximable), then we call $H$ a {\bf j-morphism} ({\bf m-morphism}).
\end{dfn}
We will show in Lemma~\ref{lem:proximitymorphismidealsfilters} that proximity morphisms preserve important structure of the proximity lattice, i.e., its round ideals and filters. See Subsection~\ref{sec:filtersandideals} for details.

\begin{rem}[On the direction of morphisms and weak isomorphisms]\label{rem:directionmorphisms}
Compared to \cite{Jung1996}, our morphisms are going in the opposite direction. We have made this choice because we want the category of algebras to be {\it dually} equivalent to the category of spaces; this way, the dual equivalence between the categories of join-strong distributive proximity lattices and stably compact spaces directly generalizes the well-known Stone duality between distributive lattices and spectral spaces. Because we kept the definition of the objects `proximity lattice' as in \cite{Jung1996}, the following results will now necessarily look slightly unnatural. Of course, the choice of direction of morphisms is ultimately a matter of taste, since the morphisms in the category are relations.
\end{rem}

\begin{lem}[\cite{Jung1996}, Section 7]
Join-strong proximity lattices with j-morphisms form a category $\JSPL$. More precisely,

\begin{enumerate}
\item The relational composition of two j-morphisms is again a j-morphism.

\item If $(\L,R)$ is a join-strong proximity lattice then $R^{-1} : (\L,R) \to (\L,R)$ is a j-morphism which acts as the identity for the composition.
\end{enumerate}
\end{lem}

Of course, we also have a category $\MSPL$ of meet-strong proximity lattices with m-morphisms.

We now also get two categories of doubly strong proximity lattices, namely the full subcategory $\DSPL_j$ of $\JSPL$ and the full subcategory $\DSPL_m$ of $\MSPL$.

We will use the terms {\bf j-isomorphic} and {\bf m-isomorphic} to indicate that proximity lattices are isomorphic in the sense of category theory. That is, $(\L,R)$ is j-isomorphic to $(\M,S)$ if there exist j-morphisms $\Phi : (\L,R) \to (\M,S)$ and $\Psi : (\M,S) \to (\L,R)$ such that $\Phi \circ \Psi = R^{-1}$ and $\Psi \circ \Phi = S^{-1}$. Note that the existence of a j- or m-isomorphism does {\it not} imply that the underlying lattices $\L$ and $\M$ are isomorphic.

Denote the category of lattices with lattice homomorphisms by $\Lat$. We then have a functor $\F : \Lat \to \JSPL$, as follows.
\begin{prop}\label{prop:Fisfunctor}
\begin{enumerate}
\item Let $\L$ be a lattice. Then $(\L,\leq)$ is a join-strong proximity lattice. 

\item Let $h : \L \to \M$ be a function between lattices. Define the relation $\F(h) : (\L, \leq) \to (\M,\leq)$ by $a\,\F(h)\,b$ iff $h(a) \geq b$. Then $h$ is a homomorphism if and only if $\F(h)$ is a j-morphism.

\item If $h : \L \to \M$ and $k: \M \to \N$ are lattice homomorphisms, then $\F(k \circ h) = \F(h) \circ \F(k)$.
\end{enumerate}
\end{prop}
\begin{proof}
\begin{enumerate}
\item The axiom of join-strongness becomes trivial when the relation is reflexive. The other axioms of a proximity lattice reduce to simple facts about lattice operations.

\item It is easy to check that $\F(h)$ is a proximity morphism if and only if $h$ is a meet-preserving function. We now show that, for any order-preserving function $h$, $\F(h)^{-1}$ is join-approximable if and only if $h$ preserves finite joins.

Suppose $\F(h)^{-1}$ is join-approximable. Take any finite subset $A$ of $\L$ and put $c := h(\bigvee A)$. Then $\bigvee h[A] \leq c$ since $h$ is order-preserving. We prove that $c \leq \bigvee h[A]$. Since we clearly have $\bigvee A \, \F(h) \, c$, we can pick $B \subseteq \F(h)[A]$ such that $c \leq \bigvee B$, by the join-approximability of $\F(h)^{-1}$. Then, for any $b \in B$, there is $a \in A$ such that $b \leq h(a)$. In particular, $b \leq \bigvee h[A]$. We conclude that $c \leq \bigvee B \leq \bigvee h[A]$. 

Conversely, suppose $h$ preserves finite joins. Take any finite subset $A$ of $\L$ and $c \in \M$ such that $\bigvee A \, \F(h) \, c$, i.e., $c \leq h(\bigvee A)$. Now put $B := h[A]$. Then it is clear that $B \subseteq \F(h)[A]$, and since $h$ preserves finite joins we have $c \leq h(\bigvee A) = \bigvee h[A] = \bigvee B$, so $c \leq \bigvee B$, as required.

\item For $a \in \L$ and $c \in \N$, we have $k(h(a)) \geq c$ iff there exists $b \in \M$ such that $k(b) \geq c$ and $h(a) \geq b$: the witness for the left-to-right direction is $b := h(a)$, and for the right-to-left direction we use that $k$ is order-preserving. Hence, $\F(kh) = \F(h) \circ \F(k)$.\qedhere
\end{enumerate}
\end{proof}

Of course, similar results hold for the assignment $\L \mapsto (\L,\leq)$ viewed as a functor into the category $\MSPL^\op$. In particular, we have that a function $h: \L \to \M$ between lattices is a homomorphism if and only if the relation $\F'(h)$ defined by $b \, \F'(h) \, a$ iff $h(a) \leq b$ is an m-morphism.

\subsection{Co-compact dual}\label{subsec:cocompact}
We already noted that the definition of proximity lattices is self-dual, and, moreover, that the opposite of a join-strong proximity lattice is a meet-strong proximity lattice. This order duality reflects certain properties of stably compact spaces, which seem to have been part of folklore for a while. These properties were summarized in Jung \cite{Jung2004}, from where we now briefly recall the results that we will need for our discussion. In fact, a large part of the theory of ordered compact spaces, which underlies these results, was already developed in the 1950's by Nachbin \cite{Nachbin1965}.
\begin{lem}[\cite{Jung2004}, Lemma 2.8]
In a stably compact space, any intersection of compact saturated sets is compact.
\end{lem}
Thus, for a stably compact space $X$, the collection $\tau^d := \{X \setminus K \ | \ K \in \KS(X)\}$ is a topology on $X$. We call the space with underlying set $X$ and topology $\tau^d$ the {\bf co-compact dual}\footnote{The co-compact dual first appeared as a general topological construction in De Groot \cite{DeGroot1967}. For that reason, some authors refer to it as the ``de Groot dual''.} of $X$, and denote it by $X^d$. 
\begin{thm}[\cite{Jung2004}, Theorem 2.12]
Let $X$ be a stably compact space. Then $X^d$ is stably compact, and $(X^d)^d$ is equal to $X$. In particular, the opens of $X$ are precisely the complements of compact saturated sets of $X^d$.
\end{thm}
\begin{exa}
If $(\E,S)$ is a compsat-basis presentation of a stably compact space $X$, consider the lattice $\D := \{X \setminus e \ | \ e \in \E\}$, ordered by inclusion, and define $R$ on $\D$ by $X \setminus e \, R \, X \setminus e'$ iff $e'Se$. Then $(\D,R)$ is an open-basis presentation of the stably compact space $X^d$, $\E^\op \isom \D$, and $R = S^{-1}$ modulo the lattice isomorphism.

\end{exa}

\begin{rem}[Regarding singly strong vs. doubly strong]
The relation between join-strong vs. meet-strong proximity lattices is clear from the previous example: a meet-strong representation of a stably compact space $X$ corresponds to a join-strong representation of the co-compact dual $X^d$. 

We thus observe that the {\it doubly} strong proximity lattices from Jung and S\"underhauf \cite{Jung1996} simultaneously represent {\it both} the space $X$ and its co-compact dual $X^d$. This is the reason why \cite{Jung1996} needed a rather complicated construction, involving pairs of open and compact sets, in order to obtain the representing lattice from a space (cf. Example~\ref{exa:doublerep} above). By contrast, we will simply use open-basis presentations to represent a space $X$ by a proximity lattice, which will not be doubly strong, but only join-strong. We thus separate the issue of representing $X$ from representing its co-compact dual $X^d$.

To make the same point differently: if one aims to represent the {\it bitopological space} $(X, \tau, \tau^d)$, then we do believe doubly strong proximity lattices are the right choice. Indeed, this is the approach taken in a more general setting in the preprint of Jung and Moshier \cite{Jung2006}. In the present paper, however, we take a `monotopological' perspective. One important reason to pursue this is that here we aim to understand {\it continuous} functions between stably compact spaces, whereas Jung and Moshier \cite{Jung2006} are constrained to {\it bicontinuous} functions (i.e., functions whose inverse image preserves open sets {\it and} compact saturated sets).
\end{rem}

\begin{rem}[Regarding join-strong vs. meet-strong]
In the rest of this paper, we will mainly use open-basis presentations and join-strong proximity lattices with j-morphisms, instead of compsat-basis presentations and meet-strong proximity lattices. This choice is of course completely arbitrary: it is clear that the algebraic theories of join-strong vs. meet-strong proximity lattices are essentially the same. Therefore, in most of what follows, we only give the results for join-strong proximity lattices and j-morphisms, and merely note that similar, order-dual, results hold for meet-strong proximity lattices and m-morphisms.
\end{rem}

\subsection{Round ideals and round filters}\label{sec:filtersandideals}
We have not yet explained how Jung and S\"underhauf \cite{Jung1996} recover a space from its presentation as a proximity lattice; we will discuss this in Section~\ref{sec:duality}. For now, it suffices to say that \cite{Jung1996} generalises Stone duality for distributive lattices and spectral spaces by generalising the notion of (prime) filters and ideals.

\begin{dfn}
A non-empty subset $I \subseteq L$ of a proximity lattice $(\L,R)$ is called a {\bf round ideal} (sometimes $R$-ideal), if it is
\begin{itemize}
\item {\bf $R$-downward closed}: for any $b \in I$, if $aRb$, then $a \in I$.
\item {\bf $R$-updirected}: for any $a, b \in I$, there is $c \in I$ such that $aRc$ and $bRc$.
\end{itemize}
Dually, a {\bf round filter} (or $R$-filter), is a subset $F$ of $L$ which is $R$-upward closed and $R$-downdirected.
\end{dfn}

The following alternative characterisation of round ideals and round filters is a bit less conceptual, but more useful in practice. This characterisation was originally given as the definition in \cite{Jung1996}.
\begin{lem}\label{lem:idealschar}
Let $(\L,R)$ be a proximity lattice. 

A subset $I \subseteq L$ is a round ideal if and only if $R^{-1}[I] = I$ and $I$ contains finite joins of its elements. A subset $F \subseteq L$ is a round filter if and only if $R[F] = F$ and $F$ contains finite meets of its elements. 

In particular, round ideals and round filters are always lattice ideals and lattice filters, respectively.
\end{lem}
\begin{proof}
The arguments are simple manipulations using the axioms for a proximity lattices and are very similar to those given in Section 3 of \cite{Jung1996}.
\end{proof}

We can now also give a useful characterisation of proximity morphisms, which shows why they are an interesting and natural class of morphisms: they `lift' to round ideals and filters, in the following sense. 
\begin{lem}\label{lem:proximitymorphismidealsfilters}
Let $(\L,R)$ and $(\M,S)$ be proximity lattices, and $T$ a relation from $L$ to $M$. The following are equivalent.
\begin{enumerate}
\item $T$ is a proximity morphism,
\item For all $a \in L$, $T[a]$ is a round ideal, and for all $b \in M$, $T^{-1}[b]$ is a round filter.
\end{enumerate}
Furthermore, it follows from these conditions that the map $T[\cdot]$ sends round ideals to round ideals, and the map $T^{-1}[\cdot]$ sends round filters to round filters.
\end{lem}
\begin{proof}
Immediate from the definitions and Lemma~\ref{lem:idealschar}.
\end{proof}

The round ideals of a distributive proximity lattice $(\D,R)$, ordered by inclusion, form a frame, which we denote by $R\Idl(\D)$ and the round filters, ordered by inclusion, form a dual frame (\cite{Jung1996}, Theorem 11). The round ideals and round filters are central to recovering a space from its basis presentation, as the following example indicates.

\begin{exa}
Let $X$ be a stably compact space, and $(\D,R)$ an open-basis presentation of $X$. Then the frame of round ideals of $\D$ is isomorphic to the frame of opens of $X$. The isomorphism sends a round ideal $I$ of basic open sets to the open set $U := \bigcup_{d \in I} d$ of $X$.

Dually, if $(\E,S)$ is a compsat-basis presentation of $X$ then the dual frame of round filters is isomorphic to the dual frame of compact saturated sets of $X$. The isomorphism sends a round filter $F$ to the compact saturated subset $K := \bigcap_{e \in F} e$ of $X$; also see the next subsection.

It follows, because $X$ is sober, and the open-basis presentation is join-strong, that the space of points of the frame of round ideals is isomorphic to $X$. The points of this frame correspond precisely to the prime round filters of the open-basis presentation of $X$. We come back to this point in Example~\ref{exa:primerfiltersarepoints} in Section~\ref{sec:duality}.
\end{exa}

The lattice $R\Idl(\L)$ is categorically a natural object to consider, because it gives a right adjoint to the functor $\F$ from Proposition~\ref{prop:Fisfunctor}. 

\begin{prop}\label{prop:Fhasrightadjoint}
Define the functor $\G : \JSPL \to \Lat$ by sending a join-strong proximity lattice $(\M,S)$ to the lattice $S\Idl(\M)$ of round ideals of $\M$ and a j-morphism $T : (\L,R) \to (\M,S)$ to the homomorphism $\G(T) : \G(\L,R) \to \G(\M,S)$ given by sending a round ideal $I$ to the round ideal $T[I]$. Then $\G$ is a functor which is right adjoint to $\F$.
\end{prop}

\begin{proof}
It is not hard to check that if $T : (\L,R) \to (\M,S)$ is a j-morphism and $I$ is a round ideal, then $T[I]$ is a round ideal. It is now easy to show that $\G$ is a functor.

To show that $\G$ is right adjoint to $\F$, we need to show that the j-morphisms $(\L,\leq_\L) \to (\M,S)$ naturally correspond to the lattice homomorphisms $\L \to S\Idl(\M)$. 

Given a j-morphism $T : (\L, \leq_\L) \to (\M,S)$, let $f_T$ be the function $\L \to S\Idl(\M)$ which sends $a \in L$ to $T[a]$. It is not hard to prove from the proximity axioms for $T^{-1}$ that $f_T$ is well-defined and preserves finite meets. To show that $f_T$ preserves finite joins, one uses the property that $T^{-1}$ is join-approximable, cf. the proof of Proposition~\ref{prop:Fisfunctor}(2).

Conversely, given a homomorphism $f: \L \to S\Idl(\M)$, we define the relation $T_f \subseteq L \times M$ by $a\,T_f\,b$ iff $b \in f(a)$. The fact that $T_f$ is a j-morphism again follows straightforwardly from the assumption that $f$ is a homomorphism into $S\Idl(\M)$.

The assignments $T \mapsto f_T$ and $f \mapsto T_f$ now constitute a natural bijection between the sets $\JSPL(\F(\L), (\M,S))$ and $\Lat(\L, \G(\M,S))$, so $\F$ is left adjoint to $\G$.
\end{proof}

Note that the category of lattices is in an adjunction with $\JSPL$, but in a {\it dual} adjunction with $\MSPL$, that is, in an adjunction with $\MSPL^\op$. This is a consequence of the direction of the morphisms, which is of course ultimately a matter of definition (cf. Remark~\ref{rem:directionmorphisms}).

\begin{rem}[On increasing proximity lattices]\label{rem:increasing}
As we mentioned right after our definition of proximity lattice, Jung and S\"underhauf \cite{Jung1996} stress that it is not necessary to assume that the relation $R$ of a join-strong distributive proximity lattice $(\L,R)$ is increasing (i.e., contained in the lattice order $\leq$). However, making this assumption does not change the category, up to equivalence: 

\begin{prop}\label{prop:increasingisom}
Every join-strong proximity lattice $(\L,R)$ is j-isomorphic to the {\it increasing} join-strong proximity lattice $(R\Idl(\L),\ll)$, where $\ll$ is the way-below relation in the complete lattice of round ideals of $\L$.
\end{prop} 

\begin{proof}
One may calculate that the way-below relation on $R\Idl(\L)$ says, for round ideals $I$ and $J$, that $I \ll J$ iff there exists $d \in J$ such that $I \subseteq R^{-1}[d]$. 

The j-isomorphism is given by the j-morphisms $\Phi : (\L,R) \to (R\Idl(\L),\ll)$ and $\Psi : (R\Idl(\L),\ll) \to (\L,R)$ defined by $a\,\Phi\,I$ iff $I \ll R^{-1}[a]$ and $I \, \Psi \, a$ iff $a \in I$. It is not hard, but a bit tedious, to check that $\Phi$ and $\Psi$ are indeed j-morphisms. To conclude, note that $\Phi \circ \Psi = R^{-1}$ and $\Psi \circ \Phi = \, \ll^{-1}$.
\end{proof}

Increasing proximity lattices are easier to understand than general proximity lattices. For example, we have the following fact.
\begin{prop}\label{prop:increasingorder}
Suppose $(\L,R)$ is an increasing join-strong proximity lattice. Then $R$ is reflexive if and only if $R$ is equal to the lattice order $\leq_\L$ of $\L$.
\end{prop}
\begin{proof}
The `if' direction is clear. For `only if', note that we already have $R \, \subseteq \, \leq_\L$ since $(\L,R)$ is assumed to be increasing. For the inclusion $\leq_\L \, \subseteq \, R$, suppose $a \leq_\L b$. Then $a \land b = a$, and since $R$ is reflexive we have $aRa$, so $a\,R\,a\land b$. From the proximity axiom for $\land$, we conclude that $aRb$.
\end{proof}
We will come back to the property of reflexivity and how it can make the theory of proximity lattices collapse in Proposition~\ref{prop:charreflexive}, after we introduce the canonical extension in the next section.
\end{rem}

\section{Canonical extensions of proximity lattices}\label{sec:canext}
Canonical extensions are an alternative way to obtain information about the points of the dual space of a frame $\O$, without explicitly referring to the dual space and therefore avoiding the use of the Axiom of Choice. Using canonical extensions, we can employ the concrete, geometric kind of reasoning from traditional topology in a point-free context. More precisely, the canonical extension $\C$ is a complete lattice which abstractly represents the power set (up-set) lattice of the set of points of an (ordered) topological space. 

Canonical extensions were first introduced for Boolean algebras, whose duals are Stone spaces, in \cite{Jonsson1952}, and then generalized to distributive lattices, whose duals are spectral spaces, in \cite{Gehrke1994}. We show here that this construction can be generalized to distributive proximity lattices, whose duals are stably compact spaces. In fact, we are able to present the material in this section without the assumption that the proximity lattices involved are distributive, for reasons analogous to the observations made in \cite{Gehrke2001} for lattices. Later, in Section~\ref{subsec:connection}, we will show that in the case of a {\it distributive} proximity lattices, where a relatively simple dual space is available, the canonical extension is exactly the lattice of up-sets of its dual space. 

Non-distributive lattices also have dual spaces, as described by Urquhart \cite{Urquhart1978}, but these are considerably more complicated than the dual spaces of distributive lattices. Nonetheless, the result that the canonical extension of a lattice is exactly the complete lattice naturally associated with the dual space still goes through (cf. \cite{Gehrke2001}, Remark 2.10). This raises the natural question whether the duality for distributive proximity lattices could be extended to a duality for proximity lattices. We will leave this as a question for further research, also see the Conclusion.

\subsection{Definition}
We first briefly recall the definition of a canonical extension of a lattice $\L$. 

We call a homomorphism $h : \L \to \C$, where $\C$ is a complete lattice, an {\bf extension} of $\L$.

Given an extension $h : \L \to \C$, we call an element $u \in \C$ a {\bf filter element} (or, also, {\bf closed}), if it can be written as $\bigwedge h[F]$ for some filter $F$ of $\L$, and we call it an {\bf ideal element} (or, also, {\bf open}), if it can be written as $\bigvee h[I]$ for some ideal $I$ of $\L$.

\begin{dfn}\label{def:canextlat}
Let $\L$ be a lattice. We call an extension $h : \L \to \C$ a {\bf canonical extension} of $\L$ if
\begin{enumerate}
\item (dense) For all $u, v \in \C$, if $u \nleq v$, then there exist a filter element $x$ and an ideal element $y$ such that $x \leq u$, $v \leq y$, and $x \nleq y$ in $\C$.
\item (compact) For any subsets $S, T \subseteq L$ such that $\bigwedge h[S] \leq \bigvee h[T]$ in $\C$, there are finite sets $S' \subseteq S$ and $T' \subseteq T$ such that $\bigwedge S' \leq \bigvee T'$ in $\L$.
\end{enumerate}
\end{dfn}
Note that it follows from condition (2) that $h$ is injective. This will not be the case in our more general setting of proximity lattices.

For a proximity lattice version of these definitions, we parametrize the conditions in $R$, as follows.

\begin{dfn}\label{def:canextplat}
Let $(\L,R)$ be a proximity lattice and $h : \L \to \C$ an extension.

We call $u \in \C$ a {\bf round ideal element} (or {\bf round-open}) if there is a round ideal $I$ of $L$ such that $u = \bigvee h[I]$. Dually, $u$ is a {\bf round filter element} (or {\bf round-closed}) if there is a round filter $F$ of $L$ such that $u = \bigwedge h[F]$. We denote the set of round ideal elements of the extension $h$ by $\I_R^h(\C)$, and the set of round filter elements by $\F_R^h(\C)$ (or, when the map $h$ is fixed, we just write $\I_R(\C)$ and $\F_R(\C)$).

We say a function $h : \L \to \C$ is a {\bf $\pi$-canonical extension} of the proximity lattice $(\L,R)$ if
\begin{enumerate}
\item ($R$-dense) For all $u, v \in \C$, if $u \nleq v$, then there exist a round filter element $x$ and a round ideal element $y$ such that $x \leq u$, $v \leq y$, and $x \nleq y$ in $\C$.
\item ($R$-compact) For any subsets $S, T \subseteq L$ such that $\bigwedge h[R[S]] \leq \bigvee h[R^{-1}[T]]$ in $\C$, there are finite sets $S' \subseteq S$ and $T' \subseteq T$ such that $\bigwedge S' R \bigvee T'$ in $\L$.
\item ($R$-join-preserving) For all $a \in \L$, $h(a) = \bigvee \{ h(b) \ | \ bRa \}$.
\end{enumerate}
Dually, $k : \L \to \C$ is a {\bf $\sigma$-canonical extension} of the proximity lattice $(\L,R)$ if it satisfies items 1 and 2 above, and 3':
\begin{enumerate}
\item[3'.] ($R$-meet-preserving) For all $a \in \L$, $k(a) = \bigwedge \{k(b) \ | \ aRb\}$.
\end{enumerate}
\end{dfn}
Note that if an extension $h$ is $R$-join-preserving or $R$-meet-preserving, it follows in both cases that $h$ is {\bf $R$-increasing}, i.e., for all $a, b \in \L$, if $aRb$, then $h(a) \leq h(b)$.

Before showing existence and uniqueness of the canonical extensions in the presence of strongness axioms, we now give some useful alternative characterisations of $R$-denseness and $R$-compactness. 
The reader may recognize these as the proximity-lattice versions of usual lattice-theoretical facts, and the proofs are straight-forward generalisations of these proofs.

\begin{prop}\label{prop:Rcompactchar}
The following are equivalent for any $R$-increasing extension $h: \L \to \C$.
\begin{enumerate}
\item The extension $h$ is $R$-compact,
\item For every round filter $F$ and round ideal $I$ of $\L$ such that $\bigwedge h[F] \leq \bigvee h[I]$ in $\C$, we have $F \cap I \neq \emptyset$.
\end{enumerate}
\end{prop}
\begin{proof}
Again, the direction (1) $\Rightarrow$ (2) is the easier one: the definition of $R$-compactness gives finite subsets $S'$ of $F$ and $T'$ of $I$ such that $\bigwedge S' \, R \, \bigvee T'$. We then have $\bigvee T' \in I$ because $I$ is an ideal. Also, $\bigwedge S' \in F$ since $F$ is a filter, and then, since $R[F] \subseteq F$, we also have $\bigvee T' \in F$. Hence $\bigvee T' \in F \cap I$.

For the direction (2) $\Rightarrow$ (1), let $S$ and $T$ be subsets of $\L$ such that $\bigwedge h[R[S]] \leq \bigvee h[R^{-1}[T]]$. Consider the sets $F := \{a \in L \ | \ \exists S' \subseteq_\omega S : \bigwedge S' \, R \, a \}$ and $I := \{a \in L \ | \ \exists T' \subseteq_\omega T : a \, R \, \bigvee T'\}$. It is not hard to see that $F$ is a round filter and $I$ is a round ideal. Note also that $R[S] \subseteq F$, and hence $\bigwedge h[F] \leq \bigwedge h[R[S]]$ in $\C$. Similarly, $\bigvee h[R^{-1}[T]] \leq \bigvee h[I]$, so that $\bigwedge h[F] \leq \bigvee h[I]$. By assumption, we can now pick $a \in F \cap I$. From the definitions of $F$ and $I$, we can now pick $S' \subseteq_\omega S$ and $T' \subseteq_\omega T$ such that $\bigwedge S' \, R \, a \, R \, \bigvee T'$, so that $\bigwedge S' \, R \, \bigvee T'$.\qedhere
\end{proof}

\begin{prop}\label{prop:Rdensechar}
The following are equivalent for any extension $h : \L \to \C$.
\begin{enumerate}
\item The extension $h$ is $R$-dense,
\item For any $u \in \C$, $u = \bigvee \{x \ | \ u \geq x \in \F_R^h(\C)\}$ and $u = \bigwedge \{y \ | \ u \leq y \in \I_R^h(\C)\}$.
\end{enumerate}
\end{prop}
\begin{proof}
This is a simple rewriting of the definition of $R$-dense.
\end{proof}

\subsection{Existence and uniqueness}
In this section, we will present the canonical extension of a join-strong proximity lattice as a lattice of Galois-closed sets, and show that it is unique up to isomorphism. For this, we first recall some elementary facts about polarities and Galois connections that we will need. We refer the reader to \cite{Gehrke2006} and \cite{Gehrke2010} for details.

\renewcommand{\G}{\mathcal{G}}
\renewcommand{\P}{\mathbb{P}}
A {\bf polarity} is a triple $(X,Y,Z)$ where $X$ and $Y$ are sets and $Z \subseteq X \times Y$. Any polarity gives rise to a pair of functions $(l_Z, r_Z)$ between the posets $\Po(X)$ and $\Po(Y)$, where $l_Z : \Po(X) \to \Po(Y)$ sends $u \subseteq X$ to $\{y \ | \ \forall x \in u : xZy\}$ and $r_Z : \Po(Y) \to \Po(X)$ sends $v \subseteq Y$ to $\{x \ | \ \forall y \in v : xZy\}$. This pair of functions forms a Galois connection, i.e, an adjoint pair from $\Po(X)^\op$ to $\Po(Y)$, since $l_Z(u) \supseteq v$ iff $u \subseteq r_Z(v)$. The composite $c_Z := r_Z \circ l_Z$ is therefore a closure operator on $\Po(X)$, and we denote by $\C := \G(X,Y,Z)$ the complete lattice of closed sets, i.e., $u \subseteq X$ such that $c_Z(u) = u$. We have maps $f : X \to \C$ and $g: Y \to \C$ which are given by $x \mapsto c_Z(\{x\})$ and $y \mapsto r_Z(\{y\})$, respectively.

\begin{thm}\label{thm:polarities}
Let $(X,Y,Z)$ be a polarity. 

\begin{enumerate}
\item The complete lattice $\C := \G(X,Y,Z)$ has the following properties.
\begin{enumerate}
\item For any $u \in \C$, $u = \bigvee \{f(x) \ | \ x \in X, \; f(x) \leq u\}$,
\item For any $u \in \C$, $u = \bigwedge \{g(y) \ | \ y \in Y, \; u \leq g(y)\}$,
\item For any $x \in X$, $y \in Y$, we have $f(x) \leq g(y)$ iff $xZy$.

Moreover, it follows from (a)--(c) that
\item For $x_1, x_2 \in X$, $f(x_1) \leq f(x_2)$ iff $\forall y \in Y (x_2 \, Z \, y \rightarrow x_1 \, Z \, y)$,
\item For $y_1, y_2 \in Y$, $g(y_1) \leq g(y_2)$ iff $\forall x \in X ( x \, Z \, y_1 \rightarrow x \, Z \, y_2)$,
\item For $y \in Y$, $x \in X$, $g(y) \leq f(x)$ iff $\forall x' \in X, y' \in Y ( x' \, Z \, y \land x \, Z \, y' \rightarrow x' \, Z \, y')$.
\end{enumerate}

\item If $\C'$ is a complete lattice and $f' : X \to \C'$, $g' : Y \to \C'$ are functions such that properties (a)--(c) from (1) also hold for $\C'$, $f'$ and $g'$, then there is a unique complete lattice isomorphism $\phi : \C' \to \C$ such that $\phi \circ f' = f$ and $\phi \circ g' = g$.

\item Let $\Q = (X \sqcup Y, \preccurlyeq)$ be the pre-order defined by items (c)--(f) of (1). Then the Dedekind-MacNeille completion $\C''$ of $\Q$, together with the natural inclusion maps of $f'' : X \to \C''$ and $g'' : Y \to \C''$, satisfies (a)--(c) of item (1), and hence, in particular, it is uniquely isomorphic to $\C$.
\end{enumerate}
\end{thm}

\begin{proof}
See, for example, Section 2 of \cite{Gehrke2006}.
\end{proof}

In order to construct the $\pi$-canonical extension of a join-strong proximity lattice, we now associate a polarity $(X,Y,Z)$ to a proximity lattice $(\L,R)$, as follows. 

Let $X$ be the set of round filters of $\L$ and $Y$ the set of round ideals of $\L$. Define the relation $Z$ from $X$ to $Y$ by $F \, Z \, I$ iff $F \cap I \neq \emptyset$. Let $\C := \G(X,Y,Z)$ be the associated complete lattice, and let $h : \L \to \C$ be the function given by $h(a) := g(R^{-1}[a])$. 

We will now show in a few steps that $h : \L \to \C$ is indeed a $\pi$-canonical extension of $(\L,R)$.

Note first of all that $h(a) = \{F : F \cap R^{-1}[a] \neq \emptyset\} = \{F : a \in F\}$, and hence in particular that $h$ is $R$-increasing. Also, $h : \L \to \C$ really is an extension:

\begin{lem}\label{lem:hishom}
If $(\L,R)$ is a join-strong proximity lattice, then $h : \L \to \C$ defined above is a homomorphism.
\end{lem}
\begin{proof}
Since $a \wedge b \in F$ iff $a \in F$ and $b \in F$, it is clear that $h$ preserves binary meets. Also, $h(\top_\L) = \top_\C$ because any round filter contains $\top_\L$, and $h(\bot_\L) = c_Z(\emptyset) = \bot_\C$.

To show that $h$ preserves binary joins, we will need the following equivalent formulation of join-strongness: for any round filter $F$, if $a \vee b \in F$, then there are $a'$, $b'$ such that $a'Ra$, $b'Rb$, and $a' \vee b' \in F$. To see that this condition is equivalent to join-strongness, notice that it is sufficient because any set $R[x]$ is a round filter, and it is also not hard to see that it is necessary.

Since $h$ preserves meets, it preserves order, so we only need to show that $h(a \lor b) \leq h(a) \lor h(b)$. 

Take any $F \in h(a \lor b)$, i.e., $a \lor b \in F$. We need to show that $F \in c_Z(h(a) \cup h(b)) = h(a) \vee h(b)$. Take an arbitrary $I \in l_Z(h(a) \cup h(b))$. Pick $a'$ and $b'$ as in the above reformulation of join-strongness. Then $a \in R[a']$, so $R[a'] \cap I \neq \emptyset$, and therefore $a' \in I$. Similarly, $b' \in I$. We conclude that $a' \lor b' \in I$, so $F \cap I \neq \emptyset$, as we needed to show.
\end{proof}

We now would like to identify the round filter and round ideal elements in the extension $h : \L \to \C$, and the following lemma will help us to do so. We believe this lemma could be viewed as a consequence of a more general fact about the construction of a complete lattice out of a polarity, where the sets $X$ and $Y$ in the polarity themselves have additional lattice structure, as is the case here. For our purposes, we just prove it directly.
\begin{lem}
\label{lem:hpreservesfiltersandideals}
For any round filter $F$, $f(F) = \bigwedge h[F]$, and for any round ideal $I$, $g(I) = \bigvee h[I]$. Hence, $f[X] \subseteq \F_R^h(\C)$ and $g[Y] \subseteq \I_R^h(\C)$.
\end{lem}
\begin{proof}
Take a round filter $F$. If $a \in F$, then $F \, Z \, R^{-1}[a]$, so $f(F) \leq g(R^{-1}[a]) = h(a)$. So $f(F) \leq \bigwedge h[F]$. For the other inequality, write $u := \bigwedge h[F]$. Then $u = \bigvee \{f(F') \ | \ f(F') \leq u\}$. So, to show that $u \leq f(F)$, take an arbitrary $F'$ with $f(F') \leq u$. Then $f(F') \leq h(a)$ for all $a \in F$, so $F \subseteq F'$. Therefore, any round ideal $I$ which intersects $F$ intersects $F'$, in other words, $f(F') \leq f(F)$.

The proof of the second statement is dual.
\end{proof}

By definition, any round filter element of $\C$ is the meet of the $h$-image of some round filter. Combining this with the above Lemma, we see that the image of $X$ under $f$ is {\it equal} to the set of round-closed elements, and the image of $Y$ is equal to the set of round-open elements.

\begin{prop}\label{prop:existencesigmaextension}
Let $(\L,R)$ be a join-strong proximity lattice. Then $h : \L \to \C$ is a $\pi$-canonical extension of $\L$.
\end{prop}
\begin{proof}
We showed that $h$ is a homomorphism in Lemma~\ref{lem:hishom}. We check that $h$ has the remaining properties.
\begin{enumerate}
\item $R$-dense.

By Lemma~\ref{lem:hpreservesfiltersandideals}, it suffices to show that the image of $X$ join-generates $\C$ and the image of $Y$ meet-generates $\C$.  Items (a) and (b) of Theorem~\ref{thm:polarities}(i) say precisely this.

\item $R$-compact.

Item (c) of Theorem~\ref{thm:polarities}(i) yields that for any polarity $(X,Y,Z)$, if $x \in X$ and $y \in Y$, then $x \leq y$ in $\G(X,Y,Z)$ if and only if $xZy$. In particular, in our case, if $F$ is a round filter and $I$ is a round ideal, then $\bigwedge F \leq \bigvee I$ in $\G(X,Y,Z)$ implies that $F \cap I \neq \emptyset$, which, by Proposition~\ref{prop:Rcompactchar}, is equivalent to $R$-compactness.

\item $R$-join-preserving.

Take any $a \in \L$. We need to show that $h(a) = \bigvee\{h(b) \ | \ bRa\}$.

If $bRa$, then any $F$ containing $b$ contains $a$, so $h(b) \leq h(a)$, so the join is below $h(a)$.

Conversely, if $F \in h(a)$, that is, $a \in F$, then there is some $b \in F$ such that $bRa$. We conclude that $F \in h(b)$, which is below the join.\qedhere
\end{enumerate}

\end{proof}

We now also prove uniqueness of the $\pi$-extension. 
\begin{prop}\label{prop:uniquenesspiextension}
If $(\L,R)$ is a join-strong proximity lattice and $h' : \L \to \C'$ is a $\pi$-canonical extension of $\L$, then there exists a complete lattice isomorphism $\phi : \C' \to \C$ such that $\phi \circ h' = h$.
\end{prop}
\begin{proof}
The homomorphism $h'$ induces a function $f' : X \to \C'$ defined by $f'(F) := \bigwedge_{a \in F} h'(a)$ and $g' : Y \to \C'$ defined by $g'(I) := \bigvee_{a \in I} h'(a)$. 

It follows from $R$-denseness that $f'[X]$ join-generates $\C'$ and $g'[Y]$ meet-generates $\C'$. It follows from $R$-compactness of $h'$ that $f'(F) \leq g'(I) \implies F \cap I \neq \emptyset$, and the other implication holds by the definition of $f'$ and $g'$. 

Therefore, by item (ii) of Theorem~\ref{thm:polarities}, there is a unique isomorphism $\phi : \C' \to \C$ such that $\phi \circ f' = f$ and $\phi \circ g' = g$. Since $h'$ is $R$-join-preserving, we have $h'(a) = \bigvee \{ h'(b) \ | \ bRa\} = g'(R^{-1}[a])$, so we deduce from $\phi \circ g' = g$ that $\phi \circ h' = h$.
\end{proof}

The same existence and uniqueness results hold for meet-strong proximity lattices and their $\sigma$-extensions (see item (2) of the following remark for a sketch of the proof). We denote the canonical extensions of a proximity lattice $(\L,R)$, if they exist, by $h : (\L,R) \to (\L,R)^\pi$ and $k : (\L,R) \to (\L,R)^\sigma$.

\begin{rem}\label{rem:canonicalextension}
\begin{enumerate}
\item
Our construction generalizes the algebraic construction of the canonical extension of a lattice $\L$, in the following sense. In the proximity lattice $(\L,\leq_\L)$, where the proximity relation is equal lattice order, the round filters and round ideals are simply the lattice filters and lattice ideals. Our construction then exactly gives the usual construction of the canonical extension of a lattice.
\item
We constructed the $\pi$-extension of a join-strong proximity lattice $(\L,R)$. In order to prove that a meet-strong proximity lattice $(\M,S)$ has a $\sigma$-extension, consider the join-strong proximity lattice $(\M^\op, S^{-1})$. Let $h : \M^\op \to \C$ be a $\pi$-extension of $(\M^\op, S^{-1})$. Then $k : \M \to \C^\op$, defined by $k(m) := h(m)$, is a $\sigma$-extension of $(\M,S)$.

Alternatively, more explicitly, one may consider the polarity $(S\Idl(\M),S\Filt(\M),Z)$, where the relation $Z$ is defined as before. Let the map $k: \M \to \G(S\Idl(\M),S\Filt(\M),Z)^\op$ be given by $a \mapsto g(S[a])$. Then, to be able to show that $k$ is a homomorphism, one needs to assume that $(\M,S)$ has the meet-strong property: the situation is order-dual to that in the proof of Lemma~\ref{lem:hishom}. The rest of the proof that $k$ is a $\sigma$-extension is analogous to the proof of Proposition~\ref{prop:existencesigmaextension}. 

\item Gehrke and Vosmaer \cite{Gehrke2009} express the canonical extension of a lattice as a {\it dcpo presentation} (also see \cite{Jung2008}). The same can be done for our $\pi$- and $\sigma$-canonical extensions of proximity lattices, via a straightforward generalisation of the methods used in \cite{Gehrke2009}. In that work, the big advantage in presenting the canonical extension via dcpo presentation was that it shed new light on the preservation of inequalities: known results about dcpo presentations and dcpo algebras were applied to obtain powerful results about the preservation of inequalities in the canonical extension. We expect that similar methods would apply to our setting, if one were to study the canonicity of inequations in proximity lattices. We mainly leave this as a topic for further research, although we will discuss the extensions of proximity lattice morphisms in Section~\ref{sec:extmorphisms}.
\end{enumerate}
\end{rem}

If a proximity lattice $(\L,R)$ is doubly strong, then the two canonical extensions $h: (\L,R) \to (\L,R)^\pi$ and $k : (\L,R) \to (\L,R)^\sigma$ exist. Note that the complete lattices $\G(X,Y,Z)$ and $\G(Y,X,Z^{-1})^\op$ which were used to define the $\pi$- and $\sigma$-extension, respectively, are isomorphic. However, the extension maps $h : \L \to \G(X,Y,Z)$ and $k : \L \to \G(Y,X,Z^{-1})^\op$ are not always the same. We have an easy characterization of when they do coincide.

\begin{prop}\label{prop:charreflexive}
If $(\L,R)$ is a doubly strong proximity lattice, then the following are equivalent: 
\begin{enumerate}
\item The $\sigma$-extension $k : (\L,R) \to (\L,R)^\sigma$ is also a $\pi$-extension of $(\L,R)$,
\item There exists an isomorphism $\phi : (\L,R)^\pi \to (\L,R)^\sigma$ such that $\phi \circ h = k$, 
\item $R$ is reflexive.
\end{enumerate}
\end{prop}

\begin{proof}
The direction (1) to (2) follows directly from the uniqueness of the $\pi$-extension (Proposition~\ref{prop:uniquenesspiextension}).

For (2) $\implies$ (3), note that because $(\L,R)^\sigma$ is now also a $\pi$-extension, we get, in the concrete representation of $(\L,R)^\sigma$ as $\G(X,Y,Z)$, that for any $a \in \L$, that $h(a) = g(R^{-1}[a]) = \bigwedge \{g(R^{-1}[b]) : aRb\}$. Observe that the round filter $R[a]$ is in the meet, so that $R[a] \in g(R^{-1}[a])$, which implies that $a \in R[a]$.

For (3) $\implies$ (1), note that the requirements of $R$-join-preserving and $R$-meet-preserving both become equivalent to $R$-increasingness in the case where $R$ is reflexive.
\end{proof}

Recall from Proposition~\ref{prop:increasingorder} that for {\it increasing} proximity lattices, the relation being reflexive is equivalent to it being equal to the lattice order. In this case, the $\sigma$- and $\pi$-canonical extension collapse to the usual canonical extension of a lattice.

We will come back to the relation between the two canonical extensions of a doubly strong proximity lattice in Proposition~\ref{prop:spectralcase}, after discussing the dualities for proximity lattices. In the distributive case, we will see that $R$ being reflexive is equivalent to $(\L,R)$ being j-isomorphic to a proximity lattice of the form $(\M,\leq_\M)$, where $\leq_\M$ is the lattice order on $\M$.

\subsection{Extending maps}\label{sec:extmorphisms}
As we mentioned in the Introduction, the power of canonical extensions for logic comes from their ability to deal with additional operations on lattices (morphisms) in a uniform way (cf. \cite{Gehrke2000}, \cite{Gehrke2001}). Thus, in this section, we want to extend proximity morphisms to the canonical extensions of the proximity lattices.

We fix the following setting for the rest of this section. Let $T : (\L,R) \to (\M,S)$ be a proximity morphism between join-strong proximity lattices and let $h_\L : (\L,R) \to (\L,R)^\pi$ and $h_\M : (\M,S) \to (\M,S)^\pi$ be the $\pi$-canonical extensions of $(\L,R)$ and $(\M,S)$. Additional assumptions on $T$, where needed, will be mentioned in the statements of the results.

We now define the {\bf $\pi$-extension} of $T$, a map from $(\L,R)^\pi$ to $(\M,S)^\pi$ which {\it extends} $T$, in a sense to be made precise below.

We first argue where $T^\pi$ should send round ideal elements of $(\L,R)^\pi$. Recall from Lemma~\ref{lem:proximitymorphismidealsfilters} that $T[I]$ is a round ideal, for any round ideal $I$. Now, for a round ideal element $y \in \I_R((\L,R)^\pi)$, we have that $I := h_\L^{-1}[\downarrow\!\!y]$ is the round ideal which is represented by $y$. We thus want $T^\pi$ to map $y$ to the ideal element in $(\M,S)^\pi$ which represents the round ideal $T[I]$. Briefly, we will define $T^\pi(y) := \bigvee h_\M[T[h_\L^{-1}[\downarrow\!\!y]]]$. Since the round ideal elements meet-generate the lattice $(\L,R)^\pi$, we now simply extend the assignment by taking meets. 

The formal definition is as follows.

\begin{dfn}
Let $T^\pi : \I_R((\L,R)^\pi) \to \I_R((\M,S)^\pi)$ be defined, for $y$ a round ideal element, by
\[T^\pi(y) := \bigvee \{h_\M(b) \ | \ b \in \M \text{ s.t. } \exists a \in \L : aTb \text{ and } h_\L(a) \leq y\} .\]

Now let $T^\pi : (\L,R)^\pi \to (\M,S)^\pi$ be the function defined by
\[ T^\pi(u) := \bigwedge \{T^\pi(y) : u \leq y \in \I_R((\L,R)^\pi)\}.\]

Dually, we could define the $\sigma$-extension of an $m$-morphism $U$ between meet-strong proximity lattices.
\end{dfn}

It is immediate from the definition that $T^\pi$ is order-preserving. We now show in what sense $T^\pi$ extends the proximity morphism $T$. 

\begin{lem}\label{lem:isextension}
For any $a \in L$, we have
\[ T^\pi(h_\L(a)) = \bigvee h_\M[T[a]].\]
\end{lem}
\begin{proof}
Note that $h_\L(a)$ is a round ideal element, so $T^\pi(h_\L(a)) = \bigvee h_\M[T[h_\L^{-1}[\downarrow\!\!h_\L(a)]]]$. Hence, it is clear that if $aTb$ then $h_\M(b) \leq T^\pi(h_\L(a))$, which shows that, in the required equality, the right hand side is below the left hand side.

For the converse inequality, we use denseness. Let $x = \bigwedge h_\M[F]$ be an arbitrary round filter element which is below $T^\pi(h_\L(a))$. Since $T[h_\L^{-1}[\downarrow\!\!h_\L(a)]]$ is a round ideal, $S$-compactness yields some $a' \in L$ and $b \in M$ such that $h_\L(a') \leq h_\L(a)$ and $a'Tb$. Then $\bigwedge T^{-1}[b] \leq h_\L(a') \leq h_\L(a) = \bigvee h_\L[R^{-1}[a]]$, and $T^{-1}[b]$ is a round filter, so $R$-compactness shows that $T^{-1}[b] \cap R^{-1}[a] \neq \emptyset$. It follows that $b \in T[a]$, so $x \leq \bigvee h_\M[T[a]]$, as required.
\end{proof}

Let us now discuss the meet-preservation of $T^\pi$.

\begin{lem}
$T^\pi$ preserves all meets of collections of round ideal elements.
\end{lem}

\begin{proof}
Let $U$ be a collection of round ideal elements of $(\L,R)^\pi$, and write $u := \bigwedge U$. We show that $T^\pi(u) = \bigwedge T^\pi[U]$. 
The inequality $T^\pi(u) \leq \bigwedge T^\pi[U]$ holds because $T^\pi$ is order-preserving.
For the converse inequality, we use denseness. Let $F$ be an arbitrary round filter such that $x_0 := \bigwedge h_\M[F] \leq \bigwedge T^\pi[U]$. We show that $x_0 \leq T^\pi(u)$.

Fix $y \in U$. We get that $x_0 = \bigwedge h_\M[F] \leq T^\pi(y) = \bigvee h_\M[T[h_\L^{-1}(\downarrow\!\!y)]]$.  By $S$-compactness (Proposition~\ref{prop:Rcompactchar}), we can pick some $b_y \in F \cap T[h_\L^{-1}(\downarrow\!\!y)]$. We then also pick $a'_y \in h_\L^{-1}(\downarrow\!\!y)$ such that $a'_yTb_y$, and, since $R \circ T^{-1} = T^{-1}$, we can pick $a_y \in \L$ such that $a_yRa'_y$ and $a_yTb_y$. We perform these steps for every $y \in U$, and thus get subsets $\{a_y : y \in U\}$ of $\L$ and $\{b_y : y \in U\}$ of $F$. 

Note that $G := \{a \in \L : \exists y_1, \dots, y_n \in U : \bigwedge_{y=1}^n a_{y_i}\,R\,a\}$ is a round filter of $\L$. Let $v := \bigwedge h_\L[G]$.  Now, since for every $y \in U$ we have $a'_y \in G$, we get $v \leq h_\L(a'_y) \leq y$, so that $v \leq \bigwedge U = u$, and hence $T^\pi(v) \leq T^\pi(u)$.

We finish the argument by showing that $x_0 \leq T^\pi(v)$. By definition of $T^\pi$, we may show that for an arbitrary round ideal $I$ such that $y_0 := \bigvee h_\L[I] \geq v$, we have $x_0 \leq T^\pi(y_0)$. By $R$-compactness, if $v \leq \bigvee h_\L[I]$, then there is $a_0 \in G \cap I$. By definition of $G$, there are $y_1, \dots, y_n \in U$ such that $\bigwedge_{i=1}^n a_{y_i} \, R a_0$. Since $a_{y_i} T b_{y_i}$ for every $i$, we have that $\bigwedge_{y=1}^n a_{y_i} T \bigwedge_{i=1}^n b_{y_i}$. We now put $b_0 := \bigwedge_{i=1}^n b_{y_i} \in F$, and get that $a_0Tb_0$.

Since $a_0 \in I$, we have $h_\L(a_0) \leq y_0$, and so, since $a_0 T b_0$, we have $h_\M(b_0) \leq T^\pi(y_0)$, by the definition of $T^\pi(y_0)$. Since $x_0 = \bigwedge h_\M[F]$, and $b_0 \in F$, we get $x_0 \leq h_\M(b_0)$, so we conclude $x_0 \leq T^\pi(y_0)$. 
\end{proof}

Using this Lemma, it is now fairly easy to show:

\begin{prop}
$T^\pi$ preserves all meets.
\end{prop}
\begin{proof}
Let $U$ be an arbitrary collection of elements from $(\L,R)^\pi$.

Put $U' := \{y \in \I_R((\L,R)^\pi) \ | \ \exists u \in U : y \geq u\}$. Note that $\bigwedge U' \leq \bigwedge U$, by $R$-denseness. We further have that $\bigwedge T^\pi[U]$ is a lower bound for the set $T^\pi[U']$, so $\bigwedge T^\pi[U] \leq \bigwedge T^\pi[U']$. Finally, by item (2), since $U'$ is a set of round ideal elements, we have that $T^\pi\left(\bigwedge U'\right) \leq \bigwedge T^\pi[U']$. Putting these inequalities together, we get
\[ \bigwedge T^\pi[U] \leq \bigwedge T^\pi[U'] \leq T^\pi\left(\bigwedge U'\right) \leq T^\pi\left(\bigwedge U\right).\qedhere\]
\end{proof}

Regarding joins, the situation is more delicate. However, we can show the following by methods which should look familiar by now.

\begin{lem}
$T^\pi$ preserves joins of up-directed collections of round ideal elements.
\end{lem}

\begin{proof}
Let $U$ be an up-directed collection in $\I_R((\L,R)^\pi)$. We need to show that $T^\pi(\bigvee U) \leq \bigvee T^\pi[U]$, as the other direction follows directly from the fact that $T^\pi$ is order-preserving.

By definition, $T^\pi(\bigvee U) = \bigvee h_\M[T[h_\L^{-1}[\downarrow\!\!\bigvee U]]]$. So we take an arbitrary $b \in T[h_\L^{-1}[\downarrow\!\!\bigvee U]]$ and show that $h_\M(b) \leq \bigvee T^\pi[U]$. Pick an $a \in h_\L^{-1}[\downarrow\!\!\bigvee U]$ such that $aTb$, and then, since $R^{-1} \circ T = T$, also pick $a'$ such that $a'Tb$ and $a'Ra$.

Now note that $\bigvee U \leq \bigvee h_\L[R^{-1}[\bigcup_{u\in U} h_\L^{-1}(\downarrow\!\!u)]]$, because $U$ consists of round ideal elements. Hence,
\[\bigwedge h_\L[R[a']] \leq h_\L(a) \leq \bigvee U \leq \bigvee h_\L[R^{-1}[\bigcup_{u\in U} h_\L^{-1}(\downarrow\!\!u)]].\]
By $R$-compactness, we can pick $t_1, \dots, t_n \in \bigcup_{u \in U} h_L^{-1}(\downarrow\!\!u)$ such that $a' R \bigvee_{i=1}^n t_i$. For each $i$, pick $u_i \in U$ such that $h_\L(t_i) \leq u_i$. Define $t_0 := \bigvee_{i=1}^n t_i$, and, since $U$ is up-directed, pick $u_0 \in U$ such that $u_0 \geq \bigvee_{i=1}^n u_i$. We then have $h_\L(t_0) \leq \bigvee_{i=1}^n u_i \leq u_0$. Now $a' R t_0$ and $a' T b$, so $t_0 T b$. We conclude that $b \in T[h_\L^{-1}(\downarrow\!\!u_0)]$. Now, by definition of $T^\pi(u_0)$, we get that $h_\M(b) \leq T^\pi(u_0)$, which is below $\bigvee T^\pi[U]$, as required.\qedhere
\end{proof}

If $T$ is a j-morphism, we can also show the following.
\begin{lem}
If $T : (\L,R) \to (\M,S)$ is a j-morphism, then $T^\pi$ preserves finite joins of round ideal elements.
\end{lem}
\begin{proof}
Suppose $Y = \{y_1, \dots, y_n\}$ is a finite subset of $\I_R((\L,R)^\pi)$. We need to show that $T^\pi(\bigvee Y) \leq \bigvee T^\pi[Y]$. 

For each $1 \leq k \leq n$, let $I_k$ be a round ideal such that $y_k = \bigvee h_\L[I_k]$. 

Take an arbitrary $b \in T[h_\L^{-1}(\downarrow\!\!\bigvee Y)]$, and pick $a \in \L$ such that $h_\L(a) \leq \bigvee Y$ and $aTb$. We need to show that $h_\M(b) \leq \bigvee T^\pi[Y]$. 

By denseness, it suffices to show that for an arbitrary round-closed element $x \leq h_\M(b)$, we have $x \leq \bigvee T^\pi[Y]$. Since $x$ is round-closed there is a round filter $F$ such that $x = \bigwedge h_\M[F]$, and then by $S$-compactness and $S$-join-preservingness there is some $c \in F \cap S^{-1}[b]$.

Note that $T^{-1}[F]$ is a round filter because $T$ is a proximity morphism, and that $a \in T^{-1}[F]$ since $aTbS^{-1}c$, so $aTc$. Thus, we get
\[ \bigwedge h_\L[T^{-1}[F]] \leq h_\L(a) \leq \bigvee Y \leq \bigvee h_\L\left[I_0\right],\]
where $I_0$ is the smallest round ideal containing $\bigcup_{k=1}^n I_k$, i.e., $I_0 := R^{-1}[\{\bigvee B \ | \ B \subseteq \bigcup_{k=1}^n I_k\}]$. By $R$-compactness, we can pick some $d \in T^{-1}[F] \cap I_0$. Pick $e \in F$ such that $dTe$, and pick $B \subseteq \bigcup_{k=1}^n I_k$ such that $d\,R\, \bigvee B$. We then get that $\bigvee B \, T \, e$, so, since $T^{-1}$ is join-approximable, there is a finite subset $A$ of $T[B]$ such that $e\, S \, \bigvee A$.

Now $\bigvee A \in F$, so $x \leq h_\M(\bigvee A)$. Moreover, $h_\M(\bigvee A) = \bigvee h_\M(A) \leq \bigvee T^\pi[Y]$, because $A \subseteq T[B] \subseteq \bigcup_{k=1}^n T[I_k]$. We conclude that $x \leq \bigvee T^\pi[Y]$, as required.
\end{proof}

At this point, we have seen that $T^\pi$ always preserves arbitrary meets and up-directed joins of round ideal elements, and also finite joins of round ideal elements in case $T$ is a j-morphism. It follows that, if $T$ is a j-morphism, then $T^\pi$ preserves arbitrary joins of round ideal elements. 

It is natural to ask if we can prove by similar methods that $T^\pi$ preserves {\it all} joins, as is possible for canonical extensions of lattices using a `restricted distributivity' law (cf. \cite{Gehrke2001}, Lemma 3.2). Although we are able to prove an analogue of that law for proximity lattices, we do not see at this point how to use it to generalize the proof from \cite{Gehrke2001} that $T^\pi$ preserves all joins. Instead, we will prove the result that $T^\pi$ preserves all joins using the duality, that we will connect with the canonical extension in the next section.

\section{Duality for stably compact spaces}\label{sec:duality}
Our description of the canonical extensions of proximity lattices can be viewed as an algebraic, point-free way of describing the duality between stably compact spaces and certain proximity lattices that was established in \cite{Jung1996}. Note in particular that we have not used the Axiom of Choice in the previous section. In the current section, we will discuss the picture of dualities that arises when we do assume the Axiom of Choice. 

We summarize the categories and equivalences that play a role in the following diagram.

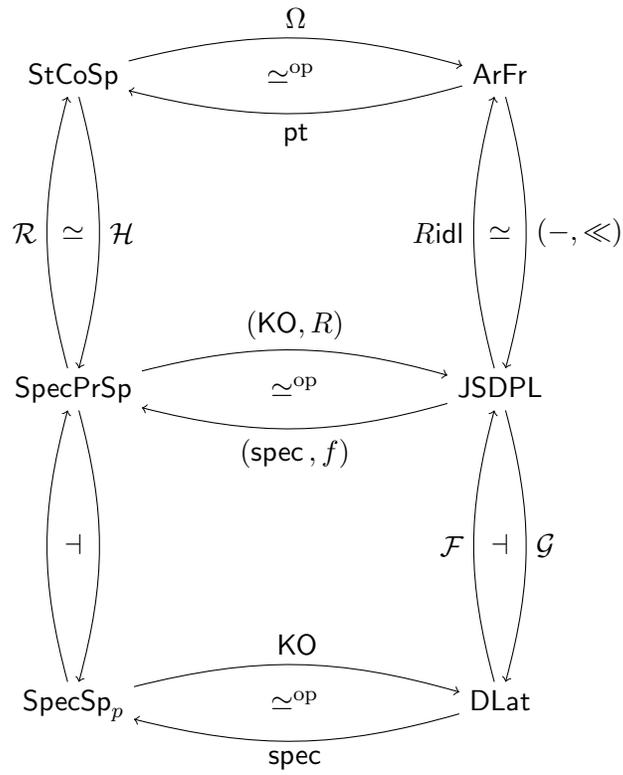
\begin{figure}[htp]\label{fig:dualitydiagram}
\centering
\begin{tikzpicture}[description/.style={fill=white,inner sep=2pt}, bend angle=15, node distance=2cm]
\matrix (m) [matrix of math nodes, row sep=4em, column sep=4em, text height=1.5ex, text depth=0.25ex]
{\cat{StCoSp} & \simeq^\op  & \cat{ArFr} \\
	\simeq		&	&	\simeq		  \\
 \cat{SpecPrSp} &  \simeq^\op & \cat{JSDPL}\\
	\lad		&	&		\lad	   \\ 
\cat{SpecSp}_p & \simeq^\op  & \cat{DLat} \\ };

\draw[->] (m-1-1) to [bend left] node[auto] {$ \Omega $} (m-1-3);
\draw[->] (m-1-3) to [bend left] node[auto] {$ \pt $} (m-1-1);
\draw[->] (m-1-1) to [bend left] node[auto] {$ \H $} (m-3-1);
\draw[->] (m-3-1) to [bend left] node[auto] {$ \sR $} (m-1-1);
\draw[->] (m-1-3) to [bend left] node[auto] {$ (-,\ll) $} (m-3-3);
\draw[->] (m-3-3) to [bend left] node[auto] {$ R\Idl $} (m-1-3);
\draw[->] (m-3-1) to [bend left] node[auto] {$ (\KO,R) $} (m-3-3);
\draw[->] (m-3-3) to [bend left] node[auto] {$ (\spec,f) $} (m-3-1);
\draw[->] (m-5-1) to [bend left] node[auto] {$ \KO $} (m-5-3);
\draw[->] (m-5-3) to [bend left] node[auto] {$ \spec $} (m-5-1);
\draw[->] (m-5-1) to [bend left] node[auto] {} (m-3-1);
\draw[->] (m-3-1) to [bend left] node[auto] {} (m-5-1);
\draw[->] (m-5-3) to [bend left] node[auto] {$ \F $} (m-3-3);
\draw[->] (m-3-3) to [bend left] node[auto] {$ \G $} (m-5-3);
\end{tikzpicture}
\caption{Dualities for stably compact and spectral spaces}
\end{figure}

The horizontal dualities in the first and last row of the diagram are well-known: $(\Omega, \pt)$ is the appropriate restriction of the well-known `geometric' duality between sober spaces and spatial frames to a duality between stably compact spaces and arithmetic frames (\cite{Abramsky1994}, Theorem 7.2.19), and the object part of $(\KO, \spec\!)$ is Stone's correspondence between spectral spaces and distributive lattices (\cite{Stone1937} and \cite{Johnstone1982}, Corollary II.3.4). The morphisms of $\DLat$ are homomorphisms, and the morphisms of $\cat{SpecSp}_p$ are the {\it perfect} or {\it bicontinuous} maps: functions for which the inverse image of a compact-open set is compact-open.

The duality in the middle row, which will connect geometric duality and Stone duality, is made explicit for the first time in this paper. We will discuss it in Section~\ref{subsec:SpPrSp}, where we show that the category of join-strong distributive proximity lattices is dually equivalent to the category $\cat{SpecPrSp}$, whose objects will be {\it spectral proximity spaces}: spectral spaces enriched with a retraction.

The vertical equivalence in the upper right corner, between arithmetic frames and join-strong distributive proximity lattices, is only a slight modification of the main achievement of \cite{Jung1996}, as it yields a finitary description (i.e., a lattice with a relation) of an infinitary algebraic representation (i.e., a frame with certain properties) of a stably compact space.

The adjunction in the lower right corner is (the restriction to distributive lattices of) the pair of functors $(\F,\G)$ that we described in Propositions~\ref{prop:Fisfunctor} and \ref{prop:Fhasrightadjoint} above: recall that $\F$ sends a lattice $\L$ to the join-strong proximity lattice $(\L,\leq)$, and $\G$ sends a proximity lattice $(\L,R)$ to the lattice of its round ideals.

Because the adjunction and equivalence in the left side of the diagram can be obtained by going through the appropriate dualities, we do not need to pay too much attention to them, but we will describe the functor $\H$ in Proposition~\ref{prop:equivalencespecprspstcosp} below, as it turns out to have a reasonably intuitive and concrete incarnation on the space side.

Before we describe the duality with spectral proximity spaces in detail, we will first isolate the essential consequence of the Axiom of Choice for proximity lattices in Section~\ref{subsec:PRFT}, which ensures that the dual space has `enough points'.
\subsection{The Prime Round Filter Theorem}\label{subsec:PRFT}
Note that up until now, the Axiom of Choice (AC) has not been used. In particular, our construction of the canonical extension of a proximity lattice did not rely on AC. However, once we want to discuss duality per se, AC must enter the picture. As is the case for Stone's duality for spectral spaces, the role of AC is to enable us to `make' the points of a space, starting from an abstract proximity lattice.

The points of the dual space of a distributive proximity lattice will be the round filters of the lattice which are {\it prime}, in the following sense.

\begin{dfn}
a round filter $F$ is called {\bf prime} if, for any finite set $A \subseteq L$, $\bigvee A \in F$ implies $A \cap F \neq \emptyset$. 
\end{dfn}

The following theorem is the relevant consequence of the Axiom of Choice in our setting.

\begin{thm}[Prime Round Filter Theorem]
Let $(\D,R)$ be a join-strong distributive proximity lattice. Let $G$ be a round filter and $J$ a round ideal such that $G \cap J = \emptyset$. Then there exists a prime round filter $F_0$ such that $F_0 \cap J = \emptyset$ and $G \subseteq F_0$.
\end{thm}
\begin{proof}
\renewcommand{\C}{\mathcal{C}}
Let $\C := \{F : F \text{ a round filter}, F \cap J = \emptyset, G \subseteq F\}$. Note that the union of a chain of round filters is a round filter. So, by Zorn's Lemma, we can take a maximal $F_0 \in \C$. It remains to show that $F_0$ is prime.

Suppose, to obtain a contradiction, that $d \lor e \in F_0$, but $d \not\in F_0$ and $e \not\in F_0$. Since $d \lor e \in F_0$, there is $a \in F_0$ such that $a\, R\, d \lor e$. Since $(\D,R)$ is join-strong, there are $d', e'$ such that $a\, R \, d' \lor e'$, $d' R d$ and $e' R e$.

Consider the set
\[ F_0 + d' := R[\{b \land d' \ : \ b \in F_0\}] \]

Note that $F_0 + d'$ is a round filter which contains $F_0$. Moreover, $F_0 \subsetneq F_0 + d'$, since $d \in F_0 + d'$, but $d \not\in F_0$ by assumption. Similarly, we have a round filter $F_0 + e'$ which properly contains $F_0$.

Since $F_0$ is a maximal element of $\C$, we conclude that both $(F_0 + d') \cap J$ and $(F_0 + e') \cap J$ are non-empty. Take elements $x$ and $y$, respectively, in the intersections. Take $b_x, b_y \in F_0$ such that $b_x \land d' \, R \, x$, and $b_y \land e' \, R \, y$. Then $b_x \land d'$ and $b_y \land e'$ are both in $J$, so $z := (b_x \land d') \lor (b_y \land e') \in J$.

On the other hand, using distributivity, we have
\vspace{-2mm}
\begin{align*}
z &= (b_x \land d') \lor (b_y \land e') \\
&= (b_x \lor b_y) \land (b_x \lor e') \land (d' \lor b_y) \land (d' \lor e').
\end{align*}

Note that the first three joins are in $F_0$ because any round filter is a lattice filter. Regarding the last join, since $a R d' \lor e'$ by construction and $a \in F_0$ we have $d' \lor e' \in F_0$. Since $F_0$ is closed under meets, we now conclude that $z \in F_0$.

But then $z \in F_0 \cap J$, giving the desired contradiction.
\renewcommand{\C}{\mathbb{C}}
\end{proof}
The prime round filters of a join-strong distributive proximity lattice will be the points of its dual space. Of course, we also have a prime round ideal theorem for meet-strong distributive proximity lattices.

\begin{dfn}
Let $(\D,R)$ be a join-strong distributive proximity lattice. The {\bf $R$-spectrum} of $(\D,R)$ is the space $X_{(\D,R)}$, whose points are prime round filters of $(\D,R)$ and whose topology is generated by the sets of the form $U_d := \{F : d \in F\}$, for $d \in \D$.
\end{dfn}

The $R$-spectrum can also be obtained via the `geometric' duality.

\begin{lem}[\cite{Jung1996}, Corollary 12]
The $R$-spectrum of a join-strong distributive proximity lattice is homeomorphic to the space of points of the arithmetic frame of round ideals.
\end{lem}

\begin{exa}\label{exa:primerfiltersarepoints}
If $X$ is a stably compact space, and $(\D,R)$ is an open-basis presentation of $X$, then a small topological argument will show that prime round filters of $\D$ correspond to completely prime filters of opens. It follows in particular that the $R$-spectrum of an open-basis presentation of $X$ is isomorphic to $X$.
\end{exa}

\subsection{Spectral proximity spaces}\label{subsec:SpPrSp}
In this subsection, we will present the duality between join-strong proximity lattices and stably compact spaces as an application of the categorical construction called `splitting by idempotents'.

We will obtain a clear understanding of the reason {\it why} join-strong distributive proximity lattices with j-morphisms are natural representations of stably compact spaces, which we believe has been present, but hidden, in the existing literature on the subject. We summarize these reasons in three steps:
\begin{enumerate}
\item Stably compact spaces are the Karoubi envelope of spectral spaces,
\item Spectral spaces are equivalent to distributive lattices with j-morphisms,
\item Join-strong distributive proximity lattices are the Karoubi envelope of distributive lattices with j-morphisms.
\end{enumerate}

We first recall the following result of Johnstone \cite{Johnstone1982}, which relates stably compact spaces to spectral spaces. Recall that a {\bf continuous retraction on a space $X$} is a continuous function $f: X \to X$ such that $f(f(x)) = f(x)$ for all $x \in X$, or, more concisely, it is an idempotent in the category $\Top$.

\begin{thm}[\cite{Johnstone1982}, Theorem VII.4.6]\label{thm:stablycompactretractofspectral}
Let $Y$ be a topological space. Then $Y$ is stably compact if and only if there exists a spectral space $X$ and a continuous retraction $f$ on $X$ such that $Y = f[X]$.
\end{thm}

Because of this theorem, one way to understand stably compact spaces is as spectral spaces, `enriched' with a continuous retraction. Although this is presumably how the idea of representing stably compact spaces by enriched distributive lattices came about, we feel this aspect of the theory has been understated in previous work. Our goal in this section is to retrace these steps and make the connection explicit.

Theorem~\ref{thm:stablycompactretractofspectral} motivates the following definitions.
\begin{dfn}
A {\bf spectral proximity space} is a pair $(X,f)$, where $X$ is a spectral space and $f$ is a continuous retraction.

A {\bf continuous proximity function} from a spectral proximity space $(X,f)$ to a spectral proximity space $(X',f')$ is a continuous function $g : X \to X'$ such that $f'g = g = gf$:

\begin{figure}[htp]
\centering
\begin{tikzpicture}[description/.style={fill=white,inner sep=2pt}, bend angle=15, node distance=2cm]
\matrix (m) [matrix of math nodes, row sep=2em, column sep=2em, text height=1.5ex, text depth=0.25ex]
{X & & X' \\
   & &   \\
 X & & X' \\ };
 \draw[->] (m-1-1) to node[auto] {$ g $} (m-1-3);
\draw[->] (m-1-1) to node[auto] {$ f $} (m-3-1);
\draw[->] (m-3-1) to node[auto] {$ g $} (m-3-3);
\draw[->] (m-1-1) to node[auto] {$ g $} (m-3-3);
\draw[->] (m-1-3) to node[auto] {$ f' $} (m-3-3);
\end{tikzpicture}
\end{figure}
We denote the category of spectral proximity spaces with continuous proximity functions by $\cat{SpecPrSp}$. 
\end{dfn}
\begin{rem}
The above definition is a particular case of a categorical construction.\footnote{Thanks to prof. Bart Jacobs for pointing us in this direction, following a presentation at the Radboud Universiteit Nijmegen on June 8, 2010.} The oldest published source for this construction seems to be Exercise B on p. 61 of \cite{Freyd1964}.

An {\bf idempotent} in a category $\C$ is an endomorphism $f$ such that $f^2 = f$. We say an idempotent $f : X \to X$ {\bf splits} if there are morphisms $r: X \to Y$ and $s: Y \to X$ such that $sr = f$ and $rs = 1_Y$. 

Given a category $\C$, we define the category $\C^s$, whose objects are pairs $(X,f)$, where $X$ is an object of $\C$ and $f$ is an idempotent on $X$, and whose morphisms $(X,f) \to (X',f')$ are $\C$-morphisms $g: X \to X'$ such that $f'g = g = gf$. The category $\C^s$ goes by many names: it can be called the {\bf Karoubi envelope}, {\bf Cauchy completion} or {\bf splitting by idempotents} of $\C$. It should be clear that the category $\cat{SpecPrSp}$ that we defined above is exactly the category $\cat{SpecSp}^s$.

The name `splitting by idempotents' for $\C^s$ is the most self-explanatory one: all idempotents in the category $\C^s$ split. Moreover, the natural functor from $\C$ to $\C^s$, which sends an object $X$ to the object $(X, 1_X)$ and is the identity on morphisms, is the universal arrow from $\C$ into a category in which all idempotents split. More explicitly, if $\C \to \B$ is a functor and all idempotents split in $\B$, then there is an (up to natural isomorphism) unique factorisation of this functor through the functor $\C \to \C^s$. It follows in particular that if $\C$ and $\D$ are equivalent categories, then so are $\C^s$ and $\D^s$. Notice that we also have that the category $(\C^s)^\op$ is isomorphic to $(\C^\op)^s$.
\end{rem}

\begin{prop}\label{prop:equivalencespecprspstcosp}
The categories $\cat{SpecPrSp}$ and $\cat{StCoSp}$ are equivalent.
\end{prop}
\begin{proof}
Given a stably compact space $Y$, by Theorem~\ref{thm:stablycompactretractofspectral} we can find a spectral space $X$ and a continuous retraction $f: X \to X$ such that $f[X] = Y$. Let $\H(Y)$ be the spectral proximity space $(X,f)$. If $h : Y \to Y'$ is a continuous function between stably compact spaces, write $\H(Y) = (X,f)$ and $\H(Y') = (X',f')$, and let $\H(h) : (X,f) \to (X',f')$ be the function defined by $\H(h) := h \circ f$. It is easy to see that $\H$ is a well-defined, full and faithful functor. It is moreover essentially surjective, using the other direction of the characterisation of Theorem~\ref{thm:stablycompactretractofspectral}: if $(X,f)$ is a spectral proximity space, then $f[X]$ is a stably compact space.
\end{proof}

By Stone duality, a distributive lattice $\D$ corresponds to the spectral space of its prime lattice filters, $X_\D$. To see how Proposition~\ref{prop:equivalencespecprspstcosp} translates into an enrichment of distributive lattices, we now recall the dual description of continuous functions between spectral spaces, which already appeared in \cite{Abramsky1994}.

\begin{prop}
Let $\D$ and $\E$ be distributive lattices and $X_\D$ and $X_\E$ the associated spectral spaces. There is a one-to-one correspondence between continuous functions $X_\D \to X_\E$ and j-morphisms $(\E, \leq_\E) \to (\D,\leq_\D)$.
\end{prop}
\begin{proof}
The correspondence sends a j-morphism $T: (\E, \leq_\E) \to (\D,\leq_\D)$ to the map $f_T : X_\D \to X_\E$ given by $F \mapsto T^{-1}[F]$. The crucial thing to note here is that $T^{-1}$ sends prime filters to prime filters if $T$ is a j-morphism. 
\end{proof}

Categorically, we now have the following result.
\begin{cor}
The category $\DLat_j$ of distributive lattices with j-morphisms between them is dually equivalent to the category of spectral spaces with continuous functions.

In particular, the Karoubi envelope of $\DLat_j$ is dually equivalent to the Karoubi envelope of $\cat{SpecSp}$.
\end{cor}

To complete our categorical considerations, we observe the following consequence of the definition of join-strong proximity lattices.
\begin{prop}
The Karoubi envelope of $\DLat_j$ is the category $\JSDPL$.
\end{prop}

Putting all the facts from this section together, we get the following chain of equivalent categories:
\[  \cat{StCoSp} \simeq \cat{SpecPrSp} = (\cat{SpecSp})^s \simeq (\DLat_j^\op)^s \isom (\DLat_j^s)^\op = \JSDPL^\op,\]
so that we have proved
\begin{thm}\label{thm:goalduality}
The category $\cat{StCoSp}$ is dually equivalent to the category $\JSDPL$.
\end{thm}

\subsection{Canonical extension via duality}\label{subsec:connection}
If we assume the Axiom of Choice, then the existence of the canonical extension can be proved via the spectrum. Specifically, we have the following result.

\begin{thm}\label{thm:canextviaduality}
Let $(\D,R)$ be a join-strong distributive proximity lattice. Let $\S$ be the complete lattice of saturated sets of $\spec(\D,R)$. Then $h : \D \to \S$, defined by $d \mapsto \{F : d \in F\}$, is a $\pi$-canonical extension of $(\D,R)$.
\end{thm}
\begin{proof}
It is not hard to see that $h$ is a homomorphism. One may further show that the round ideal elements of the extension $h : \D \to \S$ are exactly the open sets of $\spec(\D,R)$, and that the round filter elements are exactly the compact saturated sets of $\spec(\D,R)$. From this, $R$-denseness follows. For $R$-compactness, one uses the Prime Round Filter Theorem\footnote{As an anonymous referee pointed out, one could alternatively prove $R$-compactness by invoking the fact that stably compact spaces are {\it well-filtered}, following the terminology of \cite{Gierz03}, p. 147.}. The fact that $h$ is $R$-join-preserving is immediate from the definition of round filters.
\end{proof}
We also have the dual result, replacing join-strong by meet-strong and the prime round filter spectrum by the prime round ideal spectrum, cf. Remark~\ref{rem:canonicalextension}(2).

Finally, we remark on how spectral spaces fit in this picture. We already observed in Proposition~\ref{prop:charreflexive} that the $\pi$- and $\sigma$-extension of a doubly strong proximity lattice $(\L,R)$ coincide if and only if the relation $R$ is reflexive. 

If the underlying lattice is distributive, this situation relates to the dual space being spectral, as follows.

\begin{prop}\label{prop:spectralcase}
Let $(\D,R)$ be a distributive join-strong proximity lattice. The following are equivalent.
\begin{enumerate}
\item $(\D,R)$ is j-isomorphic to some distributive proximity lattice of the form $(\E,\leq_\E)$,
\item $(\D,R)$ is j-isomorphic to some distributive proximity lattice $(\E,S)$ with $S$ reflexive,
\item $\spec(\D,R)$ is a spectral space.
\end{enumerate}
\end{prop}
\begin{proof}
(i) $\implies$ (ii) is trivial.

For (ii) $\implies$ (iii), it suffices to show that if $R$ is reflexive, then $\spec(\D,R)$ is spectral. A straight-forward application of the Prime Round Filter Theorem proves that each basic open set $U_d$ is compact in this situation, so that $\{U_d\}_{d \in \D}$ is a basis of compact open sets for $\spec(\D,R)$.

For (iii) $\implies$ (i), take the distributive lattice $\E$ of compact open sets of the spectral space $\spec(\D,R)$. By Stone Duality for distributive lattices, $\spec(\D,R) \isom \spec(\E,\leq_\E)$. Hence, the proximity lattices $(\D,R)$ and $(\E,\leq)$ must be j-isomorphic by Theorem~\ref{thm:goalduality}. 
\end{proof}

Looking back at Proposition~\ref{prop:charreflexive}, we now see that in the distributive case, we can conclude something more from the assumption that the relation $R$ is reflexive.
\begin{cor}
If $(\D,R)$ is a doubly strong distributive proximity lattice and $R$ is reflexive, then there is a distributive lattice $\E$ such that $(\D,R)$ is both j- and m-isomorphic to $(\E,\leq_\E)$.
\end{cor}

Regarding the extensions of morphisms, recall that at the end of section~\ref{sec:extmorphisms}, the question whether the $\pi$-extension of a j-morphism preserves all joins had to remain open. We can now show that, for distributive lattices, the $\pi$-extension of a j-morphism $T$ is concretely realized as the inverse image map of the continuous map to which $T$ corresponds via the duality. It will follow in particular that $T^\pi$ preserves all joins and meets.

More precisely, our set-up is as follows: let $T : (\D,R) \to (\E,S)$ be a j-morphism between join-strong distributive proximity lattices, and $f_T : \spec(\E,S) \to \spec(\D,R)$ the map corresponding to $T$ under the duality, i.e., $f_T(F) := T^{-1}[F]$, for $F \in \spec(\E,S)$. 

Let $h_\D : (\D,R) \to (\D,R)^\pi$ and $h_\E : (\E,S) \to (\E,S)^\pi$ be the $\pi$-canonical extensions, which, up to isomorphism, are the complete lattices of saturated sets of the spectra, as given in Theorem~\ref{thm:canextviaduality}. 

\begin{prop}\label{prop:joinpresviadual}
In the above setting, we have for all $u \in (\D,R)^\pi$ that $T^\pi(u) = f_T^{-1}[u]$. In particular, $T^\pi$ preserves all joins and meets.
\end{prop}

\begin{proof}
For an open set $y$ of the space $\spec(\D,R)$, we have for the round ideal $I := h_\D^{-1}[\downarrow\!\!y]$ that $y = \bigvee h_\D[I]$. Hence,
\[f_T^{-1}[y] = \{F \in \spec(\E,S) : T^{-1}[F] \cap I \neq \emptyset \} = \bigcup_{e \in T[I]} h_\E(e) = \bigvee h_\E[T[h_\D^{-1}[\downarrow\!\!y]]],\]
so $f_T^{-1}[y]$ agrees with the definition of $T^\pi(y)$.

Now, for any saturated set $u$ of the space $\spec(\D,R)$, we have 
\begin{align*}
f_T^{-1}[u] &= \{F \in \spec(\E,S) \ | \ f(F) \in u\} \\
&= \{F \in \spec(\E,S) \ | \ \forall y \text{ open }: u \subseteq y \limp f(F) \in y \} \\
&= \bigwedge \{f_T^{-1}[y] : u \leq y \in \I((\D,R)^\pi)\},
\end{align*}
where the last step follows from the proof of Theorem~\ref{thm:canextviaduality}. So the definition of $T^\pi$ completely agrees with the values of $f_T^{-1}$. 

For the second statement, simply note that the inverse image map of any continuous function between topological spaces preserves all joins and meets in the complete lattice of saturated sets.
\end{proof}

We make some final remarks about our last result. On the one hand, it shows that duality can be a powerful tool to answer questions which are algebraically difficult (cf. the proofs in Section~\ref{sec:extmorphisms}). On the other hand, using the duality we can so far only prove results in the distributive setting, whereas the results in Section~\ref{sec:extmorphisms} hold in arbitrary proximity lattices, for which no duality is available yet. Moreover, the duality makes essential use of the Axiom of Choice, which was avoided completely in Section~\ref{sec:extmorphisms}.

In this section, we have connected the duality of Jung and S\"underhauf \cite{Jung1996} with two existing finitary dual equivalences, whose object parts are the same. The first of these is the duality between $\DLat$ and $\cat{SpecSp}_p$, and is reflected in the lower part of Diagram 1. Secondly, we have shown how the duality between stably compact spaces and join-strong distributive proximity lattices comes out naturally as `the Karoubi envelope of' the duality between $\DLat_j$ and $\cat{SpecSp}$. 

\newpage
\section*{Conclusion}
In this paper, we have shown that the theory of canonical extensions, which in the past has generated powerful results for logics based on lattices, can be extended to proximity lattices. The canonical extension gives an algebraic description of the saturated sets of a stably compact space, starting from any basis presentation of the space, without using the Axiom of Choice. 

One obvious direction for future work is to apply the canonical extension to logic and, more specifically, to the theory of probabilistic programming. One of the original motivations for the work of Jung and S\"underhauf \cite{Jung1996} was to develop a `continuous' version of Domain Theory in Logical Form (cf. \cite{Jung2004}). Since the `classical' canonical extension proved to be a powerful tool in modal logic, we expect that canonical extensions of proximity lattices can prove to be useful in the study of the Multi-lingual Sequent Calculus developed in a sequence of papers by Kegelman, Moshier, Jung and S\"underhauf (e.g., \cite{Kegelmann2002}, \cite{Moshier2002}, \cite{Jung2004}).

The definition of the canonical extension for proximity lattices opens up a variety of new questions regarding the behaviour of canonical extensions of maps between proximity lattices, in the line of \cite{Gehrke1994}, \cite{Gehrke2000} and \cite{Gehrke2001}. We believe that we have only scratched the surface of what can be said about these questions, by showing some properties of the canonical extension of proximity morphisms in Section~\ref{sec:extmorphisms}. Indeed, if one were to develop a theory of canonicity for logics based on proximity lattices, then showing more such properties would be an interesting direction for further research to pursue.

An advantage of the way we presented our canonical extension in Section~\ref{sec:canext} of this paper is that it is modular, in the following sense: join-strong proximity lattices have $\pi$-extensions, meet-strong proximity lattices have $\sigma$-extensions, and a doubly strong proximity lattice has both extensions, simply because it is both join- and meet-strong. Moreover, these two extensions coincide if, and only if, the additional relation is reflexive. This sheds new light on a part of the classical theory of canonical extensions which was not so well understood in the past, namely that even though a lattice has {\it one} canonical extension, there are {\it two} ways to extend a lattice map, i.e., a $\sigma$- and a $\pi$-version. In the context of our work, we would explain this apparent anomaly by saying that a lattice in principle does have both a $\sigma$- and a $\pi$-extension, but that these two extensions coincide because of the reflexivity of the lattice order.

An important related issue in the theory of proximity lattices and their canonical extensions that we think needs to be explored further is that of order-duality. More precisely, the power of canonical extensions lies in their ability to deal with additional operations on the lattice which may be order-preserving or order-reversing. It is thus important to understand in more detail the role of order-duality, which was not the focus of this paper. For the same reason, it also seems natural to wonder whether there is a natural, more symmetrical notion of extension for proximity lattices which do not satisfy any of the `strong' axioms.

In Section~\ref{sec:duality}, we have put the Jung and S\"underhauf duality from \cite{Jung1996} in a categorical perspective. A question which remains open is whether it is possible, as it was in the classical case \cite{Urquhart1978}, to remove the requirement of distributivity on the underlying lattice. This would involve replacing the spectrum of prime round filters by the space of maximal pairs of round filters and round ideals, in a sense which we believe could be made precise in future work.

A last important issue that we would like to understand better is how essential the use of duality in Section~\ref{subsec:connection}. Concretely, we would expect that versions of Propositions~\ref{prop:spectralcase} and \ref{prop:joinpresviadual} hold in which we no longer need to refer to the dual space and therefore also avoid the Axiom of Choice and the assumption that the underlying lattices are distributive.

\section*{Acknowledgements}
Many thanks to my PhD supervisor prof. Mai Gehrke, to prof. Achim Jung and to the anonymous referee. The major part of the research reported on in this paper was done during an Erasmus-funded stay at the LIAFA laboratory at Universit\'e Paris-Diderot VII, to whom I am most grateful for their hospitality. The PhD research project of the author has been made possible by NWO grant 617.023.815 of the Netherlands Organization for Scientific Research (NWO).

\bibliographystyle{amsplain}
\bibliography{proximitylattices}

\end{document}